\documentclass[12pt,twoside,reqno]{amsart}
\usepackage[colorlinks=true,citecolor=blue]{hyperref}
\usepackage{mathptmx, amsmath, amssymb, amsfonts, amsthm, mathptmx, enumerate, color}
\usepackage{graphicx}
\usepackage{epstopdf}

\newtheorem{theorem}{Theorem}[section]
\newtheorem{corollary}{Corollary}[section]
\newtheorem{lemma}{Lemma}[section]

\theoremstyle{definition}

\newtheorem{remark}{Remark}[section]

\numberwithin{equation}{section}

\usepackage{nicematrix}
\usepackage{stmaryrd}
\SetSymbolFont{stmry}{bold}{U}{stmry}{m}{n}
\usepackage{caption}
\usepackage{subcaption}
\usepackage{tikz}
\usepackage{pgfplots}
\usepackage{bm}
\usepackage{mathrsfs}
\usepackage{soul}

\pgfplotsset{compat=1.15}
\definecolor{darkbrown}{rgb}{0.57, 0.40, 0.13}

\renewcommand{\L}{\mathfrak{L}}
\newcommand{\A} {\bm{A}}
\newcommand{\lin} {\bm{\ell}}
\newcommand{\B} {\bm{B}}
\newcommand{\E} {\bm{E}}
\newcommand  {\T}   {\mathscr T}
\newcommand  {\R}  {\mathbb R}
\newcommand  {\N}  {\mathbb N}
\newcommand{\norm}[1]    {\Vert #1 \Vert}

\begin{document}
\setcounter{page}{1}

\vspace*{1.0cm}
\title[A posteriori error estimates for variational inequalities of the second kind]
{A posteriori error estimates for variational inequalities of the second kind in an abstract framework}
\author[L.~Banz, M.~Schönauer, A.~Schröder]{Lothar Banz$^1$, Miriam Schönauer$^{1,*}$, Andreas Schröder$^1$}
\maketitle
\vspace*{-0.6cm}

\begin{center}
{\footnotesize {\it

$^1$Department of Mathematics, Paris Lodron University of Salzburg, 5020 Salzburg, Austria

}}\end{center}

\vskip 4mm {\small\noindent {\bf Abstract.}
In this paper, an abstract framework is introduced for the derivation of reliable a posteriori error estimators for variational inequalities of the second kind. Under certain conditions (local) efficiency of the error estimators can also be derived. The framework is based on the estimation of a residuum which results from an equivalent mixed formulation and from the replacement of the Lagrange multiplier by its (possibly post-processed) discrete counterpart.
The approach is applied to an idealized frictional contact problem and a Mosolov type model problem in an $hp$-finite element discretization setup. Several numerical experiments demonstrate the broad applicability of the theoretical findings.



}

\renewcommand{\thefootnote}{}
\footnotetext{ $^*$Corresponding author.
\par
E-mail addresses: lothar.banz@plus.ac.at (L.~Banz), miriam.schoenauer@plus.ac.at (M.~Schönauer), andreas.schroeder@plus.ac.at (A.~Schröder).
\par
Received September XX, 2026; Accepted Month XX, 202X. }

\section{Introduction}

Variational inequalities of the second kind play an important role in mechanical engineering. They can, for instance, be used to model frictional contact problems, problems of elastoplacity and Bingham fluid problems, see e.g.~\cite{Duvaut1976,Glowinski1984, Han2013,Hlavacek1988a,Kikuchi1988a}. Typically, variational inequalities of the second kind are characterized by the presence of a non-differentiable functional $j$, which can be resolved by the introduction of a Lagrange multiplier in a mixed formulation that is equivalent to the variational inequality, see e.g.~\cite{Chouly2023,Glowinski1976a,Han1995,Haslinger2004}.

Error control and adaptivity are well-established tools for the numerical solution of variational inequalities with finite and boundary elements. 
 Both are typically based on reliable and efficient error estimators.
There are several approaches to construct such error estimators for variational inequalities. These approaches are based on, for instance, averaging techniques, e.g.~\cite{Bartels2004}, dual-weighted residual approaches, e.g.~\cite{Blum2003,Rademacher2015,Schroder2011c}, H(div)-conforming stress approximations, e.g.~\cite{Weiss2010}, the hypercircle method, e.g.~\cite{Braess2008} or residual-based techniques, e.g.~\cite{Alberty1999a,Burg2015,Carstensen2006a,Carstensen2017, Dorsek2010,Hild2007,Krause2015,Schroder2011} and references therein. A typical ingredient for the derivation of a posteriori error estimators consists in the introduction of an auxiliary problem in the form of a variational  equation so that error estimators for variational equations can be applied, see e.g.~\cite{Veeser2001,Banz2025e,Petsche2017,schrder2011,Bammer2025e,Banz2024,Braess2005} and references therein. In particular, we refer to \cite{Banz2021} where a general framework based on this approach is proposed for variational inequalities of the first kind. 

The aim of this paper is to give an abstract framework for constructing a posteriori error estimators for variational inequalities of the second kind extending the results of \cite{Banz2021}. For this purpose, we present an abstract formulation of variational inequalities of the second kind and its equivalent mixed formulation. We prove an upper bound for the error of the mixed problem for arbitrarily given elements of the underlying primal and dual spaces in terms of the residual of the equation part of the mixed formulation and a family of functionals $\bm{E}$, which measure a consistency error and the violation of the condition imposed by the functional $j$. These given elements are, for instance, discretization solutions or approximations resulting from iterative solution schemes or from reconstruction approaches. Assuming the existence of a reliable and efficient a posteriori error estimator for the dual norm of the residual, we show upper and (local) lower a posteriori error bounds  for a general error estimator by choosing the optimal and still simple to evaluate functional $\bm{E}$.  The assumed reliable and efficient error estimator for the dual norm of the residual may be realized by introducing an auxiliary problem in the form of a variational equation that is formulated in such a way that the given primal element (as an element of a discretization space) is also the (discrete) Galerkin solution of the auxiliary problem. As we make no limitations to which elements the a posteriori error estimator can be applied, the efficiency estimate contains a reduced order error term. That implies, that the convergence rate of the error estimator is not more than once and not less than half of the convergence rate of the approximation error itself. This is a known phenomena for higher order discretizations of variational inequalities of the second kind, see e.g.~\cite{Banz2022,Banz2024,Bammer2025e}. For specific, typically conforming lowest order $h$-versions, this potential (but often numerically not observed) gap in the convergence rates can be closed, see e.g.~\cite{Carstensen2006a,Bammer2025e,Schroder2011}. Here, we also discuss conditions under which this problem does not occur.

To elucidate the theoretical findings, we discuss an idealized frictional problem and a Mosolov type model problem and introduce an $hp$-finite element discretization for these problems. We show for both model problems that they fall within the abstract framework and construct for each an a posteriori error estimator. Finally, we conduct several numerical experiments based on finite element discretizations ranging from lowest order uniform $h$-version to $hp$-adaptivity to demonstrate the applicability of the error estimators.

This paper is structured as follows: In Section~\ref{sec: framework} we introduce a general variational inequality of the second kind and its equivalent mixed formulations. Section~\ref{sec: A Posteriori} is devoted to deriving the general error estimator and establishing its upper and lower error bound. In Section~\ref{sec: application} we discuss an idealized frictional problem and a Mosolov type model problem and construct a posteriori error estimators for each of them. We numerically validate the efficiency and reliability of the a posteriori estimators derived for several $h$- and $hp$-finite element methods in Section~\ref{sec: numerics}. The final Section~\ref{app:Gemischt} contains skipped proofs that belong to the general setup in Section~\ref{sec: framework}.

\section{General Setting}\label{sec: framework}
   Let $V$ be a real Hilbert space with inner product $(\cdot,\cdot)_V$ and induced norm $\| \cdot \|_V$, and let $V^*$ be its dual space with duality pairing $\langle \cdot,\cdot \rangle_V$ and dual norm $\|\cdot\|_{V^*}$. The operator $\A: V\to V^*$ is assumed to be linear, continuous and elliptic, i.e.~for all $v\in V$
    \begin{equation}\label{eq: Contiuity & Ellipticity of A}
        \Vert \A v \Vert_{V^*}\leq C_A\Vert v\Vert_V \qquad \text{and} \qquad  C_a \Vert v\Vert_V^2 \leq \langle \A v,v\rangle_V
    \end{equation}
    for some constants $C_A,\,C_a>0$. Furthermore, let $\lin\in V^*$ be arbitrarily given and let $W$ be another real Hilbert space defined with its dual space $W^*$ and duality pairing $\langle \cdot,\cdot \rangle_W$. Let $\Omega\subset\R^d$ be a bounded domain with {$d\in\N $}. Moreover, let $G\subseteq\partial\Omega$ or $G\subseteq\Omega$. We assume $\gamma:V\to W$ and $\mathfrak{L}:W\to [L^2(G)]^k$ with $k\in\N$ to be bounded linear operators, i.e.~there exist $ C_{\gamma},C_L>0$ such that for all $v \in V$ and $w \in W$
    \begin{equation}\label{eq: Continuity of gamma and L}
        \norm{\gamma(v)}_W\leq C_{\gamma}\norm{v}_V\qquad \text{and} \qquad \norm{\L(w)}_{0,G}\leq C_L\norm{w}_W,
    \end{equation}   
    where $\norm{\cdot}_{0,G}$ denotes the $L^2$-norm on $G$ for scalar or vector valued functions.
     We further assume that $\gamma(V)$ is a Banach space equipped with the norm $\norm{\cdot}_{\gamma(V)}$, which is dense in $W$. Moreover, let $C_l>0$ be a constant such that for all $w \in W$ 
    \begin{equation}\label{eq: Lowerbound for L}
        C_l\norm{w}_W\leq\norm{\L(w)}_{0,G}.
    \end{equation}
    With $\rho\in L^{\infty}_+(G):=\{v\in L^\infty(G):v>0\}$ we define the functional $j: V \to\R$ by
    \begin{equation*}
        j(v) := \bigl(\rho\,,|\mathfrak{L}(\gamma(v))|\bigr)_{0,G}
    \end{equation*}     
    where $(\cdot,\cdot)_{0,G}$ and $|\cdot|$ denote the $L^2$-inner product on $G$ and the Euclidean norm in $\mathbb{R}^k$, respectively.
    
    We consider the following variational inequality of the second kind: Find $u\in V$ such that
    \begin{equation}\label{Variational Inequality}
        \langle \bm{A}u-\bm{\ell},v-u\rangle_V +j(v)-j(u) \geq 0
    \end{equation}
    for all $v\in V$. As $j$ is convex, continuous and proper (see Lemma~\ref{lem: j properties} in Section~\ref{app:Gemischt})
    there exists a unique solution for \eqref{Variational Inequality}, see \cite[Thm.~4.1]{Glowinski1984}.
    
     In order to state a mixed formulation we introduce the operator $\B:W\to W^*$ by
    \begin{equation}\label{eq: Definition of B}
        \langle \B w, z\rangle_W=(\L(w),\L(z))_{0,G}
    \end{equation}
    for all $w,z\in W$.  By \eqref{eq: Continuity of gamma and L}, \eqref{eq: Lowerbound for L} and the properties of the $L^2$-inner product we see that $\B$ is linear, continuous and elliptic, i.e.~for all $w\in W$
    \begin{equation}\label{eq: B cont and ell}
        \norm{\B w}_{W^*}\leq C_B\norm{w}_W \qquad \text{and} \qquad C_b\norm{w}_W^2\leq \langle \B w, w\rangle_W
    \end{equation}
    with $C_B:=C_L^2$ and $C_b:=C_l^2$. With this operator we define
    \begin{equation} \label{eq:DefLambda}
        \Lambda:=\{w\in W: \langle \B \mu, \gamma(v)\rangle_W\leq j(v) \text{ for all } v\in V\}.
    \end{equation}
     which is non-empty, closed and convex, see Lemma~\ref{lem: Lambda properties} in Section~\ref{app:Gemischt}.
    Moreover, the operators $\B$ and $\gamma$ satisfy the following inf-sup condition, which is, in particular, used in the proofs of Theorems~\ref{Existenz Gemischte Formulierung} and  \ref{thm: Weighted functional First estimation}.
    
\begin{theorem}\label{thm:inf-sup}
    \emph{There exists a constant $\beta>0$ such that for all $w\in W$}
        \begin{equation}\label{lem:dense subset norm estimation}
            \beta\norm{w}_W\leq \sup\limits_{\substack{v\in V \\ \norm{v}_V=1}}\langle \B w ,\gamma(v)\rangle_W .
        \end{equation}

\end{theorem}
    \begin{proof}
     See Section~\ref{app:Gemischt}.
    \end{proof}

%
    \begin{theorem}\label{Existenz Gemischte Formulierung}
        If $u\in V$ is the solution to \eqref{Variational Inequality}, then there exists a ${\lambda}\in\Lambda$ such that $(u,{\lambda})$ uniquely solves the mixed formulation
        \begin{subequations}\label{Mixed Formulation}
            \begin{align}
                \langle \bm{A}u,v\rangle_V &= \langle \bm{\ell}, v\rangle_V - \langle\B \lambda,\gamma(v)\rangle_W,\label{Mixed Formulation Gleichung}\\
                \langle \B(\mu&-\lambda),\gamma(u)\rangle_W \leq 0\label{Mixed Formulation Ungleichung}
            \end{align}
        \end{subequations}
        for all $v\in V$ and all $\mu\in \Lambda$.
        Conversely, if $(u,{\lambda})\in V\times \Lambda$ is a solution of the mixed formulation \eqref{Mixed Formulation}, then $u$ is a solution of \eqref{Variational Inequality}.
        Furthermore, it holds 
        \begin{equation}\label{eq: conditions on lambda and gamma u}
                j(u)= \langle \B \lambda,\gamma(u)\rangle_W.
        \end{equation}
    \end{theorem}
    \begin{proof}
        See Section~\ref{app:Gemischt}. 
    \end{proof}

\section{A Posteriori Error Estimation}\label{sec: A Posteriori}

    In this section, we derive a general error estimator of the mixed formulation \eqref{Mixed Formulation}, which is based on the combination of an estimation of the dual norm of the residual of \eqref{Mixed Formulation Gleichung} and a functional $\E$ that measures a consistency error and the violation of the condition \eqref{eq: conditions on lambda and gamma u}. 
    For the estimates we consider arbitrarily given $\tilde{u}\in V$ and $\tilde{\lambda}\in W$, which are typically some approximations of $u$ and $\lambda$, respectively. The functional $\E(\cdot;\tilde{u},\tilde{\lambda}):[L^2(G)]^k \rightarrow \mathbb{R}$ is defined as 
    \begin{equation*}
        \E(\theta;\tilde{u},\tilde{\lambda}):= \norm{{\theta}-\L(\tilde{\lambda})}_{0,G}^2+j(\tilde{u})-\bigl(\theta,\L(\gamma(\tilde{u}))\bigr)_{0,G}
    \end{equation*}
    for $\theta\in [L^2(G)]^k$.
    We note that due to the definition of $j(\cdot)$ it holds for every $\theta\in\Lambda_\rho:=\{\tau\in [L^2(G)]^k: |\tau|\leq \rho\}$ that
    \begin{equation}\label{eq: j-mu geq 0}
        j(v)- \bigl(\theta,\L(\gamma(v))\bigr)_{0,G}\geq 0
    \end{equation}
    for all $v\in V$. We also point out that $\L(\mu) \in \Lambda_\rho$ implies $\mu \in \Lambda$, whereas the opposite direction holds if $\L(\gamma(V))$ is dense in $[L^2(G)]^k$, see Lemma~\ref{lem: Lambda weak strong equivalence}. 

 
    \begin{theorem}\label{thm: Weighted functional First estimation}
        There exist constants $C_{\text{Res}}, C_r>0$ independent of $(\tilde{u},\tilde{\lambda})$ such that for all $\theta\in \Lambda_\rho$
        \begin{equation}\label{eq: First upper error estimate}
            \norm{u-\tilde{u}}_V^2+\norm{{\lambda}-\tilde{\lambda}}_{W}^2\leq C_{\text{Res}}\:\norm{\bm{A}\tilde{u}-\bm{\ell}+\B\tilde{\lambda}\circ\gamma}_{V^*}^2+C_r\:\bm{E}(\theta;\tilde{u},\tilde{\lambda}).
        \end{equation}

    \end{theorem}

    \begin{proof}
        By applying \eqref{eq: Contiuity & Ellipticity of A} and \eqref{Mixed Formulation Gleichung} as well as writing $\langle \B{\lambda}\circ\gamma,\tilde{u}\rangle_V = \langle \B{\lambda},\gamma(\tilde{u})\rangle_W$, we obtain
        \begin{align*}
            C_a\norm{u-\tilde{u}}_V^2 &\leq \langle \bm{A}(u-\tilde{u}),u-\tilde{u}\rangle_V\\
            & = \langle \bm{\ell}-\B{\lambda}\circ \gamma-\bm{A}\tilde{u},u-\tilde{u}\rangle_V\\
            & = \langle \bm{\ell} -\B\tilde{\lambda}\circ \gamma -\bm{A}\tilde{u},u-\tilde{u}\rangle_V- \langle \B{\lambda}\circ\gamma,u-\tilde{u}\rangle_V + \langle \B\tilde{\lambda}\circ\gamma,u-\tilde{u}\rangle_V.
        \end{align*}
        Then, by the linearity of $\B$ and $\gamma$, the definition of the dual norm and Young's inequality we have
        \begin{equation*}
        \aligned
             C_a\norm{u-\tilde{u}}_V^2 &\leq \norm{\bm{A}\tilde{u}-\bm{\ell}+\B\tilde{\lambda}\circ\gamma}_{V^*}\norm{u-\tilde{u}}_V +\langle \B(\tilde{\lambda}-{\lambda}),\gamma(u-\tilde{u})\rangle_W\\
            &\leq \frac{1}{4\delta} \norm{\bm{A}\tilde{u}-\bm{\ell}+\B\tilde{\lambda}\circ\gamma}_{V^*}^2 + \delta\norm{u-\tilde{u}}_V^2 + \langle \B(\tilde{\lambda}-{\lambda}),\gamma(u-\tilde{u})\rangle_W
            \endaligned
        \end{equation*}
        for all $\delta>0$.
        This implies
        \begin{equation}\label{u-v Estimate}
            \norm{u-\tilde{u}}_V^2 \leq \frac{1}{4\delta(C_a-\delta)}\norm{\bm{A}\tilde{u}-\bm{\ell}+\B\tilde{\lambda}\circ\gamma}_{V^*}^2 +\frac{1}{C_a-\delta}\langle \B(\tilde{\lambda}-{\lambda}),\gamma(u-\tilde{u})\rangle_W
        \end{equation}
        for $C_a>\delta>0$.
        Exploiting \eqref{eq: Definition of B}, \eqref{eq: conditions on lambda and gamma u} and $\lambda \in\Lambda$, we obtain for any $\theta \in \Lambda_\rho$ that
        \begin{equation*}
            \aligned
                &\langle \B(\tilde{\lambda}-{\lambda}),\gamma(u-\tilde{u})\rangle_W\\
                &\quad =  \langle \B\tilde{\lambda},\gamma(u-\tilde{u})\rangle_W -  \langle \B\lambda,\gamma(u-\tilde{u})\rangle_W + \bigl(\theta,\L(\gamma(u-\tilde{u}))\bigr)_{0,G}- \bigl(\theta,\L(\gamma(u-\tilde{u}))\bigr)_{0,G}\\
                &\quad = \bigl(\L(\tilde{\lambda})-\theta,\L(\gamma(u-\tilde{u}))\bigr)_{0,G} + \bigl(\theta-\L(\lambda),\L(\gamma(u))\bigr)_{0,G}-\bigl(\theta,\L(\gamma(\tilde{u}))\bigr)_{0,G} + \langle \B\lambda,\gamma(\tilde{u})\rangle_W\\
                &\quad \leq \bigl(\L(\tilde{\lambda})-\theta,\L(\gamma(u-\tilde{u}))\bigr)_{0,G} + j(\tilde{u}) -\bigl(\theta,\L(\gamma(\tilde{u}))\bigr)_{0,G} +\bigl(\theta,\L(\gamma(u))\bigr)_{0,G} - j(u).
            \endaligned
        \end{equation*}
        Then, as $\theta \in \Lambda_\rho$, we obtain with \eqref{eq: j-mu geq 0} and \eqref{eq: Continuity of gamma and L} that
        \begin{equation*}
        \aligned
             &\langle \B(\tilde{\lambda}-{\lambda}),\gamma(u-\tilde{u})\rangle_W\\
                &\quad \leq \bigl(\L(\tilde{\lambda})-\theta,\L(\gamma(u-\tilde{u}))\bigr)_{0,G} + j(\tilde{u}) -\bigl(\theta,\L(\gamma(\tilde{u}))\bigr)_{0,G}+\bigl(\theta,\L(\gamma(u))\bigr)_{0,G} - j(u)\\
                &\quad \leq \bigl(\L(\tilde{\lambda})-\theta,\L(\gamma(u-\tilde{u}))\bigr)_{0,G} + j(\tilde{u}) -\bigl(\theta,\L(\gamma(\tilde{u}))\bigr)_{0,G}\\
                &\quad \leq \norm{\L(\tilde{\lambda})-\theta}_{0,G}\norm{\L(\gamma(u-\tilde{u}))}_{0,G}+j(\tilde{u}) -\bigl(\theta,\L(\gamma(\tilde{u}))\bigr)_{0,G}\\
                &\quad \leq C_LC_\gamma\norm{\theta-\L(\tilde{\lambda})}_{0,G}\norm{u-\tilde{u}}_V+j(\tilde{u}) -\bigl(\theta,\L(\gamma(\tilde{u}))\bigr)_{0,G}.
                \endaligned
        \end{equation*}
        We combine the above estimate with Young's inequality, which results in
        \begin{equation}\label{Kappa Lambda duale Paarung}
            \aligned
                 \langle \B(\tilde{\lambda}-{\lambda}),\gamma(u-\tilde{u})\rangle_W
                &\leq \frac{1}{4\varepsilon}\norm{\theta-\L(\tilde{\lambda})}_{0,G}^2+\varepsilon C_L^2C_\gamma^2\norm{u-\tilde{u}}_V^2+j(\tilde{u}) -\bigl(\theta,\L(\gamma(\tilde{u}))\bigr)_{0,G}\\
                &\leq \varepsilon C_L^2C_\gamma^2\norm{u-\tilde{u}}_V^2 + \frac{1}{\varepsilon}\bm{E}(\theta;\tilde{u},\tilde{\lambda})
            \endaligned
        \end{equation}
        for all $\varepsilon\in (0,1)$ as \eqref{eq: j-mu geq 0} holds. Inserting \eqref{Kappa Lambda duale Paarung} into \eqref{u-v Estimate} yields the inequality
        \begin{equation}\label{eq: u-v intermediate estimate}
            \norm{u\!-\!\tilde{u}}_V^2 \leq \frac{1}{4\delta(C_a\!-\!\delta\!-\!\varepsilon  C_L^2C_{\gamma}^2) }\norm{\bm{A}\tilde{u}\!-\!\bm{\ell}\!+\!\B\tilde{\lambda}\circ\gamma}_{V^*}^2 +\frac{1}{\varepsilon( C_a\!-\!\delta\!-\!\varepsilon C_L^2C_{\gamma}^2)}\, \bm{E}(\theta;\tilde{u},\tilde{\lambda})
        \end{equation}
        for $0<\varepsilon<\min \{1,(C_a-\delta)(C_L^2C_{\gamma}^2)^{-1}\}$.
        Next, by the inf-sup condition of Theorem~\ref{thm:inf-sup} and \eqref{Mixed Formulation Gleichung} we obtain
        \begin{equation*}
            \aligned
                \beta \norm{\lambda-\tilde{\lambda}}_W & \leq \sup\limits_{\substack{v\in V\\ \norm{v}_V=1}}\langle \B(\lambda-\tilde{\lambda}),\gamma(v)\rangle_W\\
                &=\sup\limits_{\substack{v\in V\\ \norm{v}_V=1}}\langle \B(\lambda-\tilde{\lambda})\circ \gamma,v\rangle_V\\
                &= \norm{\B(\lambda-\tilde{\lambda})\circ\gamma}_{V^*}\\
                &=  \norm{\lin-\A u-\B\tilde{\lambda}\circ\gamma}_{V^*}\\
                &\leq \norm{\A \tilde{u}-\lin+\B\tilde{\lambda}\circ\gamma}_{V^*}+\norm{\A(u-\tilde{u})}_{V^*}.
            \endaligned
        \end{equation*}
        Then, we use \eqref{eq: Contiuity & Ellipticity of A} and have
        \begin{equation}\label{lambda-kappa estimate}
        \aligned
                \norm{\lambda-\tilde{\lambda}}_W^2 &\leq \left(\frac{1}{\beta}\right)^2 \bigl(\norm{\A \tilde{u}-\lin+\B\tilde{\lambda}\circ\gamma}_{V^*}+C_A\norm{u-\tilde{u}}_{V}\bigr)^2\\
                &\leq 2\left(\frac{1}{\beta}\right)^2\norm{\A \tilde{u}-\lin+\B\tilde{\lambda}\circ\gamma}_{V^*}^2+2\left(\frac{C_A}{\beta}\right)^2\norm{u-\tilde{u}}_V^2.
                \endaligned
        \end{equation}
        Adding \eqref{lambda-kappa estimate} to \eqref{eq: u-v intermediate estimate} yields \eqref{eq: First upper error estimate} with 
        \begin{equation*}
        \aligned
            C_{\text{Res}}\!&:=\!\frac{\beta^2\!+\!2\,C_A^2\!+\!8\delta(C_a\!-\!\delta\!-\!\varepsilon C_L^2C_\gamma^2)}{\beta^24\delta(C_a\!-\!\delta\!-\!\varepsilon  C_L^2C_{\gamma}^2 )} > 0,\qquad 
            C_r\!&:=\!\frac{\beta^2\!+\!2\,C_A^2}{\varepsilon \beta^2(C_a\!-\!\delta\!-\!\varepsilon  C_L^2C_{\gamma}^2)} > 0
        \endaligned
        \end{equation*}
        for $C_a>\delta>0$ and $\min\{(C_a-\delta)(C_L^2C_{\gamma}^2)^{-1},1\}>\varepsilon>0$.
    \end{proof}


   Theorem~\ref{thm: Weighted functional First estimation} gives an entire family of upper error estimates. It is natural to ask for the best member of that family, i.e~to choose a $\theta^*\in\Lambda_\rho$ such that $\E(\theta^*;\tilde{u},\tilde{\lambda})$ is as small as possible. 
    Using the projection operator $\Pi_{\Lambda_\rho}:[L^2(G)]^k\to\Lambda_\rho$  we conclude from \cite[Lem.~16]{Bammer2025e} that ${\theta}^*:= \Pi_{\Lambda_\rho}(\tilde{\omega})$
    with $\tilde{\omega}:=\L(\tilde{\lambda})+\frac12\L\big(\gamma(\tilde{u})\big)$ is optimal in this sense, i.e.
    \begin{equation} \label{eq:optimal_Prop_theta_star}
            \E(\theta^*;\tilde{u},\tilde{\lambda})=\min\limits_{{\theta}\in \Lambda_\rho}\E(\mu;\tilde{u},\tilde{\lambda}).
        \end{equation}
    Note that $\theta^*$ can be computed explicitly by  
    \begin{align} \label{eq:optimal_theta_star}
    \theta^*= \frac{\rho}{\max\{\rho,|\tilde{\omega}|\}}\tilde{\omega}.
    \end{align}

    \begin{remark} \label{rem:subsetMinimizer}
        We emphasize that $\E(\cdot; \tilde{u},\tilde{\lambda})$ is composed of $L^2$-inner products. Thus, we easily find that $\theta^*|_K$ also minimises
        \begin{equation*}
            \E_K(\theta; \tilde{u},\tilde{\lambda}):=\norm{{\theta}-\L(\tilde{\lambda})}_{0,K}^2+\bigl(\rho,|\L(\gamma(\tilde{u}))|\bigr)_{0,K}-\bigl(\theta,\L(\gamma(\tilde{u}))\bigr)_{0,K}
        \end{equation*}
        over $\{\theta \in [L^2(K)]^k: |\theta|\leq \rho\}$
        for any subdomain $K\subseteq G$. 
    \end{remark}


In order to derive the general error estimator
we assume that there exists a reliable and efficient error estimator $\hat{\eta}$ for $\norm{\A \tilde{u}-\lin+\B\tilde{\lambda}\circ\gamma}_{V^*}$, i.e.~there exist $\widehat{C}_E,\widehat{C}_R>0$ and some (data oscillations) $\operatorname{osc}\geq 0$ such that
    \begin{subequations}\label{Hilfschaetzer}
        \begin{align}
             \norm{\A \tilde{u}-\lin+\B\tilde{\lambda}\circ\gamma}_{V^*}^2&\leq \widehat{C}_R\:\big(\hat{\eta}^2+\operatorname{osc}^2\big)\label{Hilfsschaetzer Zuverlässig} \\
            \widehat{C}_E\:\hat{\eta}^2&\leq \norm{\A \tilde{u}-\lin+\B\tilde{\lambda}\circ\gamma}_{V^*}^2+\operatorname{osc}^2.\label{Hilfsschaetzer Effizienz}
        \end{align}
    \end{subequations}
      To provide an error estimator $\hat{\eta}$ one may consider the 
    auxiliary problem: Find $\hat{u}\in V$ such that 
    \begin{align} \label{Auxiliary Problem}
    \langle \A \hat{u},v \rangle_V = \langle \lin, v\rangle_V - \langle \B\tilde{\lambda}, \gamma(v) \rangle_W
    \end{align}
    for all $v \in V$.  It is easy to see that
    \[
        C_a \|\hat{u}-\tilde{u}\|_V  \leq \norm{\A \tilde{u}-\lin+\B\tilde{\lambda}\circ\gamma}_{V^*}\leq C_A \|\hat{u}-\tilde{u}\|_V,
    \]
    i.e.~one may choose $\hat{\eta}$ as an error estimator for $\|\hat{u}-\tilde{u}\|_V$.

     Eventually, combining $\hat{\eta}^2$ with the functional $\E$ and choosing the optimal $\theta^*$ yields the general error estimator
     \begin{equation*}
          \eta^2:=\hat{\eta}^2+\E(\theta^*;\tilde{u},\tilde{\lambda}).
    \end{equation*}
    
    Before we prove lower and upper bounds for $\eta^2$, we establish the following lemma.
  \begin{lemma}\label{lem: Friction local eff preliminary}
  Let $\L(\lambda) \in \Lambda_\rho $. Then there holds  $\L(\lambda)^\top \L(\gamma(u)) =\rho |\L(\gamma(u))|$ almost everywhere in $G$.
   \end{lemma}
  \begin{proof}
     For $\L(\lambda)\in \Lambda_\rho$ it trivially holds that
  \begin{equation}\label{idealized friction local eff hilf}
       \L(\lambda)^\top \L(\gamma(u)) \leq |\L(\lambda)| \, |\L(\gamma(u))| \leq \rho |\L(\gamma(u))|
  \end{equation}
  almost everywhere in $G$. Then, by \eqref{eq: conditions on lambda and gamma u} it holds that
  \begin{equation*}
       0=j(u)- \langle \B \lambda,\gamma(u)\rangle_W = \int_G \rho |\L(\gamma(u))|- \L(\lambda)^\top \L(\gamma(u))\,dx,
 \end{equation*}
  which in combination with \eqref{idealized friction local eff hilf} yields the assertion.
 \end{proof}

    \begin{theorem}\label{Vorstufe Rel und Eff}
        There is a constant $C_{\text{rel}}>0$ independent of $\tilde{u},\tilde{\lambda}$ and $\theta^*$ such that
        \begin{equation*}
            \norm{u-\tilde{u}}_V^2+\norm{{\lambda}-\tilde{\lambda}}_{W}^2 \leq C_{\text{rel}}\Bigl(\eta^2+\operatorname{osc}^2\Bigr).
        \end{equation*}
    \end{theorem}
 \begin{proof}
        Since $\theta^*\in \Lambda_\rho$ it follows from Theorem~\ref{thm: Weighted functional First estimation} and \eqref{Hilfsschaetzer Zuverlässig} that
        \begin{align*}
            \norm{u-\tilde{u}}_V^2+\norm{{\lambda}-\tilde{\lambda}}_{W}^2 & \leq C_{\text{Res}}\:\norm{\bm{A}\tilde{u}-\bm{\ell}+\B\tilde{\lambda}\circ\gamma}_{V^*}^2+C_r\:\bm{E}(\theta^*;\tilde{u},\tilde{\lambda}) \\
            & \leq  C_{\text{Res}}\:\widehat{C}_R\:\big(\hat{\eta}^2+\operatorname{osc}^2\big)+C_r\:\bm{E}(\theta^*;\tilde{u},\tilde{\lambda}) \\
            & \leq C_{\text{rel}}\Big(\eta^2+\operatorname{osc}^2\Big)
        \end{align*}
        with $C_{\text{rel}}:= \max\{C_{\text{Res}}\widehat{C}_R, \;C_r\}>0$.
    \end{proof}

    \begin{theorem}\label{effThm}
     Let $\L(\lambda)\in \Lambda_\rho$. Then,
        there is a constant $C_{\text{eff}}>0$ independent of $\tilde{u},\tilde{\lambda}$ and $\theta^*$ such that
        \begin{equation*}
            C_{\text{eff}}\;\eta^2\leq  \norm{u-\tilde{u}}_V^2 + \norm{{\lambda}-\tilde{\lambda}}_{W}^2 + \norm{u-\tilde{u}}_V+ \operatorname{osc}^2.
        \end{equation*}
    \end{theorem}

   \begin{proof}
         From \eqref{eq:optimal_Prop_theta_star}, \eqref{eq: conditions on lambda and gamma u}, $\lambda\in\Lambda$, \eqref{eq: Continuity of gamma and L} and Cauchy Schwarz inequality it follows that
        \begin{equation*}
        \begin{split}
            \E(\theta^*;\tilde{u},\tilde{\lambda}) &\leq \E(\L({\lambda});\tilde{u},\tilde{\lambda}) \\
            &= \norm{\L(\lambda-\tilde{\lambda})}_{0,G}^2+j(\tilde{u})-\langle \B\lambda,\gamma(\tilde{u})\rangle_W\\
            &= \norm{\L(\lambda-\tilde{\lambda})}_{0,G}^2+j(\tilde{u})-j(u)+\langle \B\lambda,\gamma(u-\tilde{u})\rangle_W\\
            &\leq \norm{\L(\lambda-\tilde{\lambda})}_{0,G}^2+j(\tilde{u}-u)+j(u)-j(u)+\langle \B\lambda,\gamma(u-\tilde{u})\rangle_W\\
            &\leq \norm{\L(\lambda-\tilde{\lambda})}_{0,G}^2+2j(\tilde{u}-u) \\
            &\leq C_L^2\norm{\lambda-\tilde{\lambda}}_W^2+2\norm{\rho}_{0,G}\norm{\L(\gamma(\tilde{u}-u))}_{0,G}\\
            &\leq C_L^2\norm{\lambda-\tilde{\lambda}}_W^2+2C_LC_\gamma\norm{\rho}_{0,G}\norm{u-\tilde{u}}_V.
         \end{split}
        \end{equation*}
         Thus, we obtain with \eqref{Hilfsschaetzer Effizienz} that
        \begin{multline*}
            \hat{\eta}^2\!+\!\bm{E}(\theta^*;\tilde{u},\tilde{\lambda})\leq\frac{1}{\widehat{C}_E}\Big(\! \norm{\A \tilde{u}\!-\!\lin\!+\!\B\tilde{\lambda}\!\circ\!\gamma}_{V^*}^2\!+\!\operatorname{osc}^2\!\Big)\!+\!C_L^2\norm{\lambda\!-\!\tilde{\lambda}}_W^2\!+\! 2C_LC_\gamma\norm{\rho}_{0,G}\norm{u\!-\!\tilde{u}}_V.
        \end{multline*}
        With \eqref{Mixed Formulation Gleichung}, triangle inequality, the definition of the dual norm $V^*$, \eqref{eq: Contiuity & Ellipticity of A}, \eqref{eq: Continuity of gamma and L} and \eqref{eq: B cont and ell} we easily find that
        \begin{align*}
        \norm{\A \tilde{u}-\lin+\B\tilde{\lambda}\circ\gamma}_{V^*} \leq C_A \norm{u-\tilde{u}}_V + C_BC_\gamma \norm{\lambda - \tilde{\lambda}}_W.
        \end{align*}
        Hence, we have
        \begin{align*}
        \eta^2 = \hat{\eta}^2+\bm{E}(\theta^*;\tilde{u},\tilde{\lambda})\leq C_{\text{eff}}^{-1}\Bigl( \norm{u-\tilde{u}}_V^2 + \norm{{\lambda}-\tilde{\lambda}}_{W}^2 + \norm{u-\tilde{u}}_V+ \operatorname{osc}^2 \Bigr)
        \end{align*}
         with        
            \begin{equation*}\label{localE}
                 C_{\text{eff}}^{-1}:= \max\left\{{2C_A^2\widehat{C}_E^{-1}}{},{2C_B^2C^2_{\gamma}\widehat{C}_E^{-1}}{}+C_L^2,2C_LC_\gamma\norm{\rho}_{0,G},{\widehat{C}_E^{-1}}{}\right\}.
             \end{equation*}
        \end{proof}

With similar arguments as in the proof Theorem~\ref{effThm} we see that $\E(\theta^*;\tilde{u},\tilde{\lambda})$ satisfies a local efficiency estimate.

\begin{corollary} \label{cor:localEffiE}  

  Let $\L(\lambda)\in \Lambda_\rho$. Then, 
        \begin{equation*}
            \bm{E}_K(\theta^*; \tilde{u},\tilde{\lambda})
            \leq \norm{\L(\lambda-\tilde{\lambda})}_{0,K}^2 +2 \norm{\rho}_{0,K}\norm{\L(\gamma(u-\tilde{u}))}_{0,K}
        \end{equation*}
        for any subdomain $K\subseteq G$.
\end{corollary}

\begin{proof}
        
        For $K\subseteq G$ it holds by Lemma~\ref{lem: Friction local eff preliminary} and the reverse triangle inequality that
        \begin{equation*}
        \aligned
            \bm{E}_K({\L({\lambda})}; \tilde{u},\tilde{\lambda})
            &= \norm{\L(\lambda-\tilde{\lambda})}_{0,K}^2+\bigl(\rho,|\L(\gamma(\tilde{u}))|\bigr)_{0,K}-\bigl(\L(\lambda),\L(\gamma(\tilde{u}))\bigr)_{0,K}\\
            &= \norm{\L(\lambda-\tilde{\lambda})}_{0,K}^2+\bigl(\rho,|\L(\gamma(\tilde{u}))|-|\L(\gamma(u))|\bigr)_{0,K}+\bigl(\L(\lambda),\L(\gamma(u-\tilde{u}))\bigr)_{0,K}\\
            &\leq \norm{\L(\lambda-\tilde{\lambda})}_{0,K}^2+\bigl(\rho,|\L(\gamma(\tilde{u}-u))|\bigr)_{0,K}+\bigl(\L(\lambda),\L(\gamma(u-\tilde{u}))\bigr)_{0,K}\\
            &\leq \norm{\L(\lambda-\tilde{\lambda})}_{0,G}^2+2\bigl(\rho,|\L(\gamma(\tilde{u}-u))|\bigr)_{0,K}\\
            &\leq \norm{\L(\lambda-\tilde{\lambda})}_{0,K}^2 +2 \norm{\rho}_{0,K}\norm{\L(\gamma(u-\tilde{u}))}_{0,K}.
            \endaligned
        \end{equation*}
        Since $\theta^*|_K$ is also the minimiser of $\bm{E}_K(\cdot; \tilde{u},\tilde{\lambda})$ over $\{\theta \in [L^2(K)]^k: |\theta| \leq \rho \}$ for any subdomain $K\subseteq G$, see Remark~\ref{rem:subsetMinimizer}, this yields the assertion. 
    \end{proof}
    {
    To guarantee that $\L(\lambda)\in\Lambda_\rho$  in Theorem \ref{effThm} and Corollary \ref{cor:localEffiE} we need to impose an additional assumption on the operator $\L \circ \gamma$.
    \begin{lemma}\label{lem: Lambda weak strong equivalence}
        Let $\L(\gamma(V))$ be a dense subset of $[L^2(G)]^k$. Then, for all $w\in W$
        \[
            w\in\Lambda \quad \Leftrightarrow \quad     \L(w)\in\Lambda_\rho.
        \]
    \end{lemma}
    \begin{proof}
         Define
        \begin{equation*}
            \Lambda_L:=\{w\in W: (\L(w),z)_{0,G}\leq (\rho,|z|)_{0,G}\;\text{ for all }\; z\in [L^2(G)]^k\}
        \end{equation*}
        and let $w\in\Lambda_L$ and $v\in V$. Then,
        \[
            \langle Bw,\gamma(v)\rangle_W=(\L(w),\L(\gamma(v))_{0,G}\leq (\rho,|\L(\gamma(v))|)_{0,G}=j(v)
        \]
        which shows $w\in \Lambda$ and, thus, $\Lambda_L\subset\Lambda$. Let $w\in \Lambda$ and assume that $w\not\in \Lambda_L$. Then, there exists a $z\in [L^2(G)]^k$ and an $\varepsilon>0$ such that 
                \begin{equation}\label{proof: epsilon lower bound}
                    \varepsilon < (\L(w),z)_{0,G}-(\rho,|z|)_{0,G}.
                \end{equation}
                Furthermore, as $\L(\gamma(V))$ is dense in $[L^2(G)]^k$ there exists a $v\in V$ such that 
                \begin{equation*}
                    \norm{\L(\gamma(v))-z}_{0,G}<\varepsilon(C_L\norm{w}_{W}+\norm{\rho}_{0,G})^{-1}.
                \end{equation*}
                Therefore, as $w\in \Lambda$ and by reverse triangle inequality we have
                \begin{equation*}
                \aligned
                    (\L(w),z)_{0,G}-(\rho,|z|)_{0,G}&=(\L(w),z-\L(\gamma(v)))_{0,G}+(\L(w),\L(\gamma(v)))_{0,G}-(\rho,|z|)_{0,G}\\
                    &\leq (\L(w),z-\L(\gamma(v)))_{0,G}+(\rho,|\L(\gamma(v))|-|z|)_{0,G}\\
                    &\leq C_L\norm{w}_W\norm{\L(\gamma(v))-z}_{0,G}+(\rho,|\L(\gamma(v))-z|)_{0,G}\\
                    &\leq (C_L\norm{w}_{W}+\norm{\rho}_{0,G})\;\norm{\L(\gamma(v))-z}_{0,G}\\
                    &<\varepsilon,
                    \endaligned
                \end{equation*}
                which is a contradiction to \eqref{proof: epsilon lower bound}. Thus, $w\in \Lambda_L$ and $\Lambda=\Lambda_L$. Next, assume that $\L(w)\in \Lambda_\rho$. For $z\in [L^2(G)]^k$ we have
                \begin{equation*}
                    (\L(w),z)_{0,G}\leq (|\L(w)|,|z|)_{0,G}\leq (\rho,|z|)_{0,G},
                \end{equation*}
                which shows that $w\in \Lambda_L$. Now, let $w\in \Lambda_L$ with $N_w:=\{x\in G: |\L(w)|> \rho\}$ and set 
                \begin{equation*}
                    v(x):=\begin{cases}  \frac{\L(w(x))}{|\L(w(x))|}, \qquad &x\in N_w\\ 0, \qquad \qquad &x\in G\setminus N_w\end{cases}.
                \end{equation*}
                Clearly, $v \in [L^2(G)]^k$ and
                \begin{equation*}
                    0\leq (\rho,|v|)_{0,G}-(\L(w),v)_{0,G}
                     =\int\limits_{N_w}\rho|v|- \L(w)^\top v\,ds_x
                     =\int\limits_{N_w} \underbrace{\rho- |\L(w)|}_{<0}\,ds_x
                \end{equation*}
                which implies that $N_w$ has zero measure. This gives $\L(w)\in \Lambda_\rho$ and, thus, completes the proof.
    \end{proof}
    }
     It follows from Theorem~\ref{effThm} that the error squared and the error itself are needed to bound the error estimator squared from above.  Thus, Theorem~\ref{Vorstufe Rel und Eff} and Theorem~\ref{effThm} imply that the convergence rate of the error estimator is not more than the convergence rate of the approximation error and not less than half of that convergence rate. {In certain applications, see e.g.~Section~\ref{sec: idealized friction}, leaving the term $\norm{\L(\gamma(u-\tilde{u}))}_{0,G}$ in the proof of Theorem~\ref{effThm} as it is and not further estimating it with \eqref{eq: Continuity of gamma and L} can lead to a better efficiency estimate.} The numerical experiments in Section~\ref{sec: numerics} suggest that the efficiency estimate  in Theorem~\ref{effThm} may be improvable, which at least
can be achieved under an additional (quite restrictive) assumption, as the following statement shows.
    \begin{corollary}
        Let $\L(\tilde{\lambda})\in \Lambda_\rho$ and
            \begin{equation}\label{eq: kappa efficiency condition}
                (\L(\tilde{\lambda}),\L(\gamma(\tilde{u}))_{0,G} = j(\tilde{u}).
            \end{equation}
        Then, there exist $C_R, C_E>0$ such that
        \begin{equation*}
            \norm{u-\tilde{u}}_V^2+\norm{{\lambda}-{\tilde{\lambda}}}_{W}^2 \leq C_R\Bigl(\hat{\eta}^2+\operatorname{osc}^2\Bigr)
        \end{equation*}
        and
        \begin{equation*}
            C_E\,\hat{\eta}^2\leq \norm{u-\tilde{u}}_V^2 + \norm{{\lambda}-{\tilde{\lambda}}}_{W}^2 + \operatorname{osc}^2.
        \end{equation*}
    \end{corollary}
    \begin{proof}
         By \eqref{eq: kappa efficiency condition} it holds that $\bm{E}(\L({\tilde{\lambda}});\tilde{u},\tilde{\lambda})=0$. Thus, $\eta = \hat\eta$. Therefore, by the arguments of Theorems~\ref{Vorstufe Rel und Eff} and \ref{effThm} we get the desired estimates.
    \end{proof}
    \begin{remark}
       We emphasize that there are cases where \eqref{eq: kappa efficiency condition} is satisfied. For instance, in \cite{Bammer2025e} an elastoplastic problem with linearly kinematic hardening is analysed, which falls into the abstract framework. Using the notation of \cite{Bammer2025e} $\gamma(\cdot)$ is defined as the projection $u:=(\mathfrak{u},\bm{p})\mapsto\bm{p}$, where the tuple consists of the displacement field $\mathfrak{u}$ and the plastic strain $\bm{p}$, and $\L(\cdot)$ is defined as the identity operator, mapping the vector space of plastic strains $Q\subset[L^2(\Omega)]^{d^2}\simeq[L^2(\Omega)]^{d\times d}$ onto itself. In \cite[Rem.~2]{Bammer2025e} it is specified under which conditions \eqref{eq: kappa efficiency condition} is fulfilled. A further example in a similar setting can be found in \cite[Sec.~4]{Schroder2011}.
    \end{remark}

    Up to now $\tilde{u}$ and $\tilde{\lambda}$ are arbitrarily given in $V$ and $W$, respectively. Clearly, an error estimation is only meaningful if $\tilde{u}$ and $\tilde{\lambda}$ result from a discretization scheme. For instance, we may choose $(\tilde{u},\tilde{\lambda})$ as $(u_N,\lambda_N) \in V_N \times \Lambda_N \subset V \times W$ such that
    \begin{subequations}
            \begin{align}
                \langle \bm{A}u_N,v_N\rangle_V &= \langle \bm{\ell}, v_N\rangle_V - \langle\B \lambda_N, \gamma(v_N)\rangle_W, \label{eq:abstract:discreteMixed1}\\
                \langle \B(\mu_N&-\lambda_N),\gamma(u_N)\rangle_W \leq 0 
            \end{align}
    \end{subequations}
   for all $(v_N,\mu_N) \in V_N \times \Lambda_N$, where 
$V_N$ is a finite dimensional subset of $V$ and $\Lambda_N$ is a non-empty, closed, convex subset of a finite dimensional set $W_N \subset W$ (and a meaningful replacement for $\Lambda$). Of course, $(\tilde{u},\tilde{\lambda})$ could also be some solver iterates $(u_N^{(i)},\lambda_N^{(i)}) \in V_N \times W_N$ that converges towards $(u_N,\lambda_N)$. Alternatively, $\tilde{u} \in V_N$ is the (iterative) solution of a discrete analogue of the variational inequality \eqref{Variational Inequality}. In this case, we may compute some $\tilde{\lambda} \in W$ in an additional post-processing step such that
\begin{align} \label{eq:abstract:postprocessingLambda}
        \langle \bm{A}\tilde{u},v_N\rangle_V = \langle \bm{\ell},v_N\rangle_V-\langle\B{\tilde{\lambda}}, \gamma(v_N)\rangle_W
 \end{align}
 for all $v_N \in V_N$.
We emphasize that \eqref{eq:abstract:discreteMixed1} implies \eqref{eq:abstract:postprocessingLambda}. Moreover, \eqref{eq:abstract:postprocessingLambda} implies that $\tilde{u}$ becomes the Galerkin approximation of $\hat{u}$ that solves the auxiliary problem \eqref{Auxiliary Problem}. The resulting Galerkin orthogonality $\langle A(\hat{u}-\tilde{u}),v_N \rangle_V =0$ for all $v_N \in V_N$ enables the construction of an efficient and reliable error estimator $\hat{\eta}$ for $\norm{\hat{u}-\tilde{u}}_V$.
The detour via $\hat{u}$ is not necessary as \eqref{eq:abstract:postprocessingLambda} also implies
$$ \norm{\A \tilde{u}-\bm{\ell} -\B{\tilde{\lambda}}\circ\gamma}_{V^*} = \sup_{v\in V\setminus\{0\}} \frac{\langle \A \tilde{u}-\bm{\ell} -\B{\tilde{\lambda}}\circ\gamma,v\rangle_V}{\norm{v}_V} = \sup_{v\in V\setminus\{0\}} \frac{\langle \A \tilde{u}-\bm{\ell} -\B{\tilde{\lambda}}\circ\gamma,v-v_N\rangle_V}{\norm{v}_V}$$
for any $v_N \in V_N$, which may be helpful in directly estimating or computing the dual norm with some localizable and implementable quantities.
We also emphasize that \eqref{eq:abstract:postprocessingLambda} is an underdetermined problem, which gives us flexibility to choose $\tilde{\lambda}$ in a suitable large discretization space, so that, for instance, $\tilde{\lambda}$ can be computed by solving some local problems as proposed in \cite[Sec.~5]{Burg2015}. Of course different choices for $\tilde{\lambda}$ lead to different values of the error estimator but also to a different discretization error in particular of $\norm{\lambda-\tilde{\lambda}}_W$.


\section{Applications of the a Posteriori Error Estimates}\label{sec: application}
    
To elucidate the theoretical results, we discuss two variational inequality problems, their $hp$-finite element discretization, and derive a posteriori error estimates for these. For this purpose, we first consider an idealized frictional problem as discussed in \cite{Schroder2011a} and later a Mosolov model problem, a scalar variant of the Bingham problem, as considered in \cite{Banz2022}. We note that in both cases  the proposed a posteriori error estimators resulting from the abstract framework coincide with the a posteriori error estimators proposed in these articles. Deriving the upper and lower estimates with the abstract framework of Section~\ref{sec: A Posteriori} is, however, considerably streamlined.

    \subsection{Common Setup}
    
        Let $\Omega\subset \R^2$ be a bounded, polygonal Lipschitz domain with boundary $\Gamma:=\partial\Omega$.
        The Dirichlet boundary part is denoted by $\Gamma_D\subseteq \Gamma$, is closed and has positive measure.
        We denote the contact boundary part by $\Gamma_C\subset \Gamma\setminus\Gamma_D$ and assume that $\overline{\Gamma}_C \cap \Gamma_D = \emptyset$ allowing us to work with standard fractional order Sobolev spaces. Moreover, the Neumann boundary part is given by $\Gamma_N:=\Gamma\setminus(\Gamma_D\cup\overline{\Gamma}_C)$.     
        Let $H^1(\Omega)$ and $H^{1/2}(\Gamma_C)$ denote the usual Sobolev spaces, for which we denote their norms by $\norm{\cdot}_{1,\Omega}$ and $\norm{\cdot}_{1/2,\Gamma_C}$, respectively.  
        Furthermore, let
        \begin{equation*}
        \aligned
            H^1_0(\Omega)&:=\{v\in H^1(\Omega): \gamma_\Gamma(v)=0 \;\text{ on }\;\Gamma\},\\
            H^1_D(\Omega)&:=\{v\in H^1(\Omega): \gamma_\Gamma(v)=0 \;\text{ on }\;\Gamma_D\},
            \endaligned
        \end{equation*}
        where $\gamma_\Gamma:H^1(\Omega)\to H^{1/2}(\Gamma)$ is the usual trace operator. We note that the operator
        \begin{equation*}
            \gamma_C:=(\gamma_\Gamma)_{|_{\Gamma_C}}:H^1_D(\Omega)\to H^{1/2}(\Gamma_C)\subsetneq L^2(\Gamma_C)
        \end{equation*}
        is linear, bounded and surjective due to the properties of $\Gamma_C$, see e.g.~\cite{Kikuchi1988a}. Furthermore, the space $H^{1/2}(\Gamma_C)$ is dense in $L^2(\Gamma_C)$.
 The restriction to scalar two dimensional problems is just for notational convenience.
 
     \subsection{Discretization with \texorpdfstring{$hp$}{hp}-Finite Elements}\label{Discretization}
            
            We denote by $\T_h$ the locally quasi-uniform and $\varrho$-shape regular triangulation of $\Omega$ into parallelograms. The restriction to parallelograms is only needed for the analysis of the arbitrarily chosen classical residual error estimator for variational equations that is used to estimate $\norm{\bm{A}\tilde{u}-\bm{\ell}+\B{\tilde{\lambda}}\circ\gamma}_{V^*}$. We also assume that  all corners of $\Gamma$ and the transition points between $\Gamma_D$, $\Gamma_C$ and $\Gamma_N$ are nodes of $\T_h$.
            Furthermore, let $\widehat{Q}:=[-1,1]^2$ be the reference element and let $F_Q:\widehat{Q}\to Q$ be an affine and bijective mapping from the reference element onto an element $Q\in\T_h$. We write $h_Q:=\operatorname{diam}(Q)$ for the local element size and set $h:=(h_Q)_{Q\in\T_h}$. Moreover, let $p_Q\geq 1$ denote the local polynomial degree and set $p:=(p_Q)_{Q\in\T_h}$. We assume the polynomial degrees of neighbouring elements to be comparable.
            For more details on such a decomposition, see e.g.~\cite{Melenk2001a}.\\
            Let $\widehat{E}:=[-1,1]$ be the reference interval and let $F_E:\widehat{E}\to E$ be an affine and bijective mapping from the reference interval onto an edge $E\in\mathcal{E}_h$, where $\mathcal{E}_h$ denotes the set of edges of $\T_h$.
            We denote by $\mathcal{E}_I$ the set of interior edges of $\T_h$ and by $\mathcal{E}_C$ and $\mathcal{E}_N$ the sets of edges belonging to the contact, Neumann boundary, respectively. Moreover, let $\mathcal{E}_Q$ be the set of edges belonging to an element $Q\in \T_h$ and define $h_E:=\operatorname{diam}(E)$ as well as $p_E:=\min\limits_{Q\in \T_h}\{p_Q: E\in \mathcal{E}_Q\}$.
            In the subsequent we use the polynomial space
$      \mathbb{P}_p(\widehat{Q}):=\operatorname{span}\big\{x_1^{k_1}x_2^{k_2}\;|\;0 \leq k_1, k_2 \leq p\} $
            and the discretization spaces
            \begin{equation*}
            \aligned
                V_{hp}&:=\{v_{hp}\in H_0^1(\Omega)\;:\; v_{hp}|_Q\circ F_Q\in \mathbb{P}_{p_Q}(\widehat{Q})\;\text{ for all }\; Q\in\mathcal{T}_h\},\\
                V_{hp}^D&:=\{v_{hp}\in H_D^1(\Omega)\;:\; v_{hp}|_Q\circ F_Q\in \mathbb{P}_{p_Q}(\widehat{Q})\;\text{ for all }\; Q\in\mathcal{T}_h\}.
                \endaligned
            \end{equation*}
            Additionally, we denote by $v^{p_Q}$ the $L^2(Q)$-projection of $v|_Q$ for some $v \in L^2(\Omega)$ onto $\{v_{hp} \in L^2(Q): v_{hp}\circ F_Q\in \mathbb{P}_{p_Q-1}(\widehat{Q}) \}$. Alike, $v^{p_E}$ is the $L^2(E)$-projection of $v|_E$ for some $v \in L^2(\Gamma)$ onto $\{v_{hp} \in L^2(E): v_{hp}\circ F_E\in \mathbb{P}_{p_E-1}(\widehat{E}) \}$.            
            Furthermore, $\llbracket \cdot \rrbracket$ denotes the usual jump function across an edge, $n_E$ denotes a unit normal of the edge $E$ that is outwards for $E \subset \Gamma$ and let the patch $\omega_Q$ be given by $\omega_Q:=\{Q'\in\T_h: \mathcal{E}_Q\cap\mathcal{E}_{Q'}\neq\emptyset\}$.\\
            We emphasize that in the following, the function $\tilde{u}_{hp}\in V_{hp}$ or $\tilde{u}_{hp}\in V_{hp}^D$ is arbitrary and takes the role of $\tilde{u}$ from Section~\ref{sec: A Posteriori}.
            \subsection{An Idealized Frictional Problem}\label{sec: idealized friction}

            The first problem is an idealized frictional problem, which is to find a function $u:\Omega \rightarrow \mathbb{R}$ such that 
            \begin{equation*}
            \aligned
                -\Delta u&= f\quad\text{ in }\quad\Omega,\\
                  u &= 0\quad\text{ on }\quad \Gamma_D,\\
                \partial_nu &= g\quad\text{ on }\quad \Gamma_N,\\
                |\partial_nu|\leq \rho \quad \text{ and} \quad \partial_n u \cdot u & = - \rho |u| \quad \text{ on }\quad \Gamma_C
                \endaligned
            \end{equation*}
            for given $f\in L^2(\Omega)$, $g\in L^2(\Gamma_N)$ and $\rho\in L^{\infty}_+(\Gamma_C)$. Its variational formulation is to find a $u\in H^1_D(\Omega)$ such that
            \begin{equation}\label{eq: Idealized Friction Inequality}
                \int\limits_{\Omega}\nabla u\,\nabla(v-u)\,dx + \int\limits_{\Gamma_C}\rho\,|\gamma_C(v)|\,ds - \int\limits_{\Gamma_C}\rho\,|\gamma_C(u)|\,ds \geq \int\limits_{\Omega}f\,(v-u)\,dx +\int\limits_{\Gamma_N}g\,\gamma_\Gamma(v-u)\,ds
            \end{equation}
            for all $v\in H^1_D(\Omega)$; see e.g.~\cite[Ch.~2]{Hlavacek1988a}.
            To establish its connection to the abstract results we set $V:=H^1_D(\Omega)$ and $\A$, $\lin$ such that
            \begin{equation*}
                \langle \A v,w\rangle_V=(\nabla v,\nabla w)_{0,\Omega},\quad \langle \lin ,v\rangle_V=(f,v)_{0,\Omega}+(g,\gamma_\Gamma(v))_{0,\Gamma_N}
            \end{equation*}
            for all $v,w\in H^1_D(\Omega)$, respectively. Furthermore, we set $G:=\Gamma_C$ and $W:= L^2(\Gamma_C)$ and let $\L:=\operatorname{id}:L^2(\Gamma_C)\to L^2(\Gamma_C)$. For the operator $\gamma:=\gamma_C$ the condition \eqref{eq: Continuity of gamma and L} is fulfilled and we note that due to the surjectivity of $\gamma_C$ it holds that $\gamma_C(V)=H^{1/2}(\Gamma_C)$ which is dense in $L^2 (\Gamma_C)$. As $W = L^2(\Gamma_C)$ the operator $\L$ clearly has the properties \eqref{eq: Continuity of gamma and L} and \eqref{eq: Lowerbound for L}.
            Finally, we define
            \begin{equation*}
                j(v):= (\rho,|\gamma_C(v)|)_{0,\Gamma_C}.
            \end{equation*}
            As $\A:H^1_D(\Omega)\to \left(H^1_D(\Omega)\right)^*$ is continuous and $H^1_D(\Omega)$-elliptic, the inequality \eqref{eq: Idealized Friction Inequality} is of form \eqref{Variational Inequality} and, thus, has a unique solution $u$.

            To give a mixed formulation of the problem, we define the operator $\B:L^2(\Gamma_C)\to L^2(\Gamma_C)$ by
            \begin{equation*}
                \langle\B w, z\rangle_W=(w,z)_{0,\Gamma_C}
            \end{equation*}
            for all $w,z \in L^2(\Gamma_C)$, i.e.~$\B=\operatorname{id}$ and set 
            \begin{equation*}
                \Lambda:=\{w \in L^2(\Gamma_C): (w,\gamma_C(v))_{0,\Gamma_C}\leq j(v)\;\text{ for all }\; v\in H^1_D(\Omega)\}.
            \end{equation*}
            Since, $\L(\gamma(V))=H^{1/2}(\Gamma_C)$ is dense in $L^2(\Gamma_C)$ it holds by Lemma~\ref{lem: Lambda weak strong equivalence} that
            \begin{equation}\label{lem: Equivalent Lambda}
                    \Lambda =\Lambda_\rho:=\{\tau \in L^2(\Gamma_C): |\tau |\leq \rho\}.
                \end{equation}
            Thus, we arrive at the mixed formulation of the idealized frictional problem: Find $u\in H^1_D(\Omega)$ and $\lambda \in \Lambda$ such that
             \begin{subequations}\label{eq: Friction Mixed Formulation}
            \begin{align}
                (\nabla u, \nabla v)_{0,\Omega} &= (f,v)_{0,\Omega}+(g,\gamma(v))_{0,\Gamma_N}-(\lambda,\gamma_C(v))_{0,\Gamma_C}, \label{eq: Friction Mixed Formulation Teil1}  \\
                (\mu&-\lambda,\gamma_C(u))_{0,\Gamma_C}\leq 0
            \end{align}
            \end{subequations}
            for all $(v,\mu)\in H^1_D(\Omega) \times \Lambda$. By Theorem~\ref{Existenz Gemischte Formulierung} this problem has a unique solution $(u,\lambda)\in H^1_D(\Omega)\times \Lambda$, which is equivalent to solving to \eqref{eq: Idealized Friction Inequality}. Obviously, it holds $\lambda = -\partial_n u$ on $\Gamma_C$ in a weak sense.
            %

        \subsubsection{A Posteriori Error Estimation} \label{sec:fric:apost}

            Recall that $\tilde{u}_{hp}\in V_{hp}^D$ is an arbitrary element.
            We choose $\tilde{\lambda} \in L^2(\Gamma_C)$ such that 
            \begin{equation} \label{eq:fric:kappa_postprocessing}
                (\tilde{\lambda},\gamma_C(v_{hp}))_{0,\Gamma_C} = (f,v_{hp})_{0,\Omega}+(g,\gamma(v_{hp}))_{0,\Gamma_N}-(\nabla \tilde{u}_{hp}, \nabla v_{hp})_{0,\Omega}
            \end{equation}
            for all $v_{hp}\in V_{hp}^D$, cf.~\eqref{eq:abstract:postprocessingLambda}. One computationally cheap possibility is to determine $\tilde{\lambda}$ as a discontinuous finite element function by solving local problems as described in \cite[Sec.~5]{Burg2015}. For another cheap possibility we refer to the description of the numerical experiments in Section~\ref{sec:numeric:fric}.
        We emphasize that  the restriction given by \eqref{eq:fric:kappa_postprocessing} for choosing $\tilde{\lambda}$ (instead of just choosing $\tilde{\lambda}$ arbitrarily) enables the use of the classical residual error estimator $\hat{\eta}$ for variational equations,  see the auxiliary problem \eqref{Auxiliary Problem} and the discussion at the end of Section~\ref{sec: A Posteriori}.
        Here, \eqref{Auxiliary Problem} is to find $\hat{u}\in V$ such that
            \begin{equation} \label{eq:fric:Aux_prob}
                (\nabla \hat{u},\nabla v)_{0,\Omega} = (f,v)_{0,\Omega}+(g,\gamma(v))_{0,\Gamma_N}-(\tilde{\lambda},\gamma_C(v))_{0,\Gamma_C}
            \end{equation}
           for all $v\in V$. Since $\tilde{\lambda}$ solves \eqref{eq:fric:kappa_postprocessing}, $\tilde{u}_{hp}$ is the discrete Galerkin solution of the auxiliary problem \eqref{eq:fric:Aux_prob}.  Hence, we can simply look up the corresponding residual a posteriori error estimator  to specify $\hat{\eta}$.
            Let the local error contribution $\hat\eta_Q^2$ with $Q\in \T_h$ be given by
            \begin{equation*}
            \aligned
                \hat\eta_Q^2&:= \frac{h_Q^2}{p_Q^2}\norm{f^{p_Q}+\Delta \tilde{u}_{hp}}_{0,Q}^2 + \sum\limits_{E\in \mathcal{E}_I\cap\mathcal{E}_Q}\frac{h_E}{2p_E}\norm{\llbracket \partial_{n_E}\tilde{u}_{hp}\rrbracket}_{0,E}^2 +  \sum\limits_{E\in \mathcal{E}_C\cap\mathcal{E}_Q} \frac{h_E}{p_E}\norm{\partial_{n_E}\tilde{u}_{hp}+\tilde{\lambda}^{p_E}}_{0,E}^2 \\
                &\qquad + \sum\limits_{E\in\mathcal{E}_N\cap\mathcal{E}_Q}\frac{h_E}{p_E}\norm{g^{p_E}-\partial_{n_E}\tilde{u}_{hp}}_{0,E}^2.
              \endaligned  
            \end{equation*}
            with the element-, edge-wise $L^2$-projections $f^{p_Q}$, $g^{p_E}$, $\tilde{\lambda}^{p_E}$ of the data $f$, $g$, $\tilde{\lambda}$, respectively. 
            The (data) oscillation term $\operatorname{osc}^2:=\sum\limits_{Q\in\T_h}\operatorname{osc}_Q^2$ is defined by the sum of the local terms
            \begin{equation*}
                \operatorname{osc}_Q^2:= \frac{h_Q^2}{p_Q^2}\norm{f-f^{p_Q}}_{0,Q}^2+\sum\limits_{E\in \mathcal{E}_C\cap\mathcal{E}_Q}\frac{h_E}{p_E}\norm{\tilde{\lambda}-\tilde{\lambda}^{p_E}}_{0,E}^2 + \sum\limits_{E\in\mathcal{E}_N\cap\mathcal{E}_Q}\frac{h_E}{p_E}\norm{g-g^{p_E}}_{0,E}^2.
            \end{equation*}
            Therewith, the total error estimator becomes 
            \begin{equation*} 
                \eta^2:=\hat{\eta}^2+\bm{E}(\theta^*;\tilde{u}_{hp},\tilde{\lambda})
            \end{equation*}     
            with
            \begin{equation*}
                \hat\eta^2:=\sum\limits_{Q\in \T_h}\hat\eta_Q^2 \quad \text{and} \quad \bm{E}(\theta^*;\tilde{u}_{hp},\tilde{\lambda}):=\norm{\theta^*-\tilde{\lambda}}_{0,\Gamma_C}^2+ j(\tilde{u}_{hp})-(\theta^*,\gamma_C(\tilde{u}_{hp}))_{0,\Gamma_C}
            \end{equation*}
            and
            \begin{equation*}
                \theta^*:=\rho\frac{\tilde{\lambda}+\frac12\gamma_C(\tilde{u}_{hp})}{\max\{\rho,|\tilde{\lambda}+\frac12\gamma_C(\tilde{u}_{hp})|\}},
            \end{equation*}
            see \eqref{eq:optimal_Prop_theta_star} and \eqref{eq:optimal_theta_star}. Note that by Remark~\ref{rem:subsetMinimizer}, $\eta^2$ is also localizable, since it holds
            $$ \eta^2 = \sum\limits_{Q\in \T_h} \eta^2_Q \quad \text{with} \quad \eta^2_Q:=  \hat\eta_Q^2 + \sum\limits_{E\in \mathcal{E}_C\cap\mathcal{E}_Q} \E_E(\theta^*;\tilde{u}_{hp},\tilde{\lambda}), $$
            where
$$ \E_E(\theta^*;\tilde{u}_{hp},\tilde{\lambda}): = \int_E (\theta^*-\tilde{\lambda})^2 + \rho |\gamma_C(\tilde{u}_{hp})| - \theta^*\gamma_C(\tilde{u}_{hp}) \, ds. $$

 \begin{theorem}\label{thm: A posteriori Friction}
   There exists a constant $C_{\text{rel}}>0$ independent of $h$, $p$, $\tilde{u}_{hp}$ and $\tilde{\lambda}$ such that
  \begin{equation} \label{eq:fric:eta_reliable}
    \norm{u-\tilde{u}_{hp}}_{1,\Omega}^2+\norm{\lambda-\tilde{\lambda}}_{0,\Gamma_C}^2\leq C_{\text{rel}}\bigl[\eta^2+\operatorname{osc}^2\bigr].
\end{equation}
 For all $\varepsilon > 0$ there exists a constant $C_{\text{eff}}(\varepsilon)>0$ independent of $h$, $p$, $\tilde{u}_{hp}$ and $\tilde{\lambda}$ such that for every $Q \in \T_h$ there holds
    \begin{align}
        C_{\text{eff}}(\varepsilon)^{-1}\eta^2_Q & \leq   p_Q^{2(1+\varepsilon)} \norm{u-\tilde{u}_{hp}}^2_{1,\omega_Q} + \sum\limits_{E\in \mathcal{E}_C\cap\mathcal{E}_Q} \norm{\gamma_C(u-\tilde{u}_{hp})}_{0,E} + \left(1 + h_E p_E^{2\varepsilon} \right)\norm{\lambda - \tilde{\lambda}}^2_{0,E} \nonumber  \\
         & \quad  + \sum_{Q \in \omega_Q} p_Q^{1+4\varepsilon} \operatorname{osc}_Q^2. \label{eq:fric:effi} 
    \end{align}
 \end{theorem}
\begin{proof}
It is well known that the residual error estimator is reliable and analogously to \cite{Melenk2001a} we obtain
 \begin{equation*}
   \|\A \tilde{u}_{hp} - \lin + \B \tilde{\lambda}  \circ \gamma \|_{(H^1_D(\Omega))^*} \leq C_A  \norm{\hat{u}-\tilde{u}_{hp}}_{1,\Omega}^2\leq \hat{C}_R\hat\eta^2
 \end{equation*}
with some constant $\hat{C}_R>0$ independent of $h$, $p$, $\tilde{u}_{hp}$ and $\tilde{\lambda}$. Thus, the reliability estimate~\eqref{eq:fric:eta_reliable} follows with Theorem~\ref{Vorstufe Rel und Eff}.

Likewise, it is well known that $\hat{\eta}^2_Q$ is locally efficient to $\norm{\hat{u}-\tilde{u}_{hp}}_{1,\omega_Q}^2$. However, here it is better to redo the efficiency proof that can be found in \cite{Melenk2001a} and directly use \eqref{eq: Friction Mixed Formulation Teil1} instead of \eqref{eq:fric:Aux_prob}. Doing this we conclude that for every $\varepsilon>0$ there exists a constant $\hat{C}(\varepsilon)>0$ independent of $h$, $p$, $\tilde{u}_{hp}$ and $\tilde{\lambda}$ such that
\begin{align*}
 \eta_Q^2   &\leq \hat{C}(\varepsilon) p_Q^{1+2\varepsilon}\left[ p_Q \norm{u-\tilde{u}_{hp}}^2_{1,\omega_Q} + \sum\limits_{E\in \mathcal{E}_C\cap\mathcal{E}_Q} \frac{h_E}{p_E}\norm{\lambda - \tilde{\lambda}}^2_{0,E} \right. \\
 & \quad \left.+ p_Q^{2\varepsilon} \frac{h_Q^2}{p_Q^2}\norm{f-f^{p_Q}}^2_{0,\omega_Q} + \sum\limits_{E\in \mathcal{E}_N\cap\mathcal{E}_Q} \frac{h_E}{p_E}\norm{g - q^{p_E}}^2_{0,E} + \sum\limits_{E\in \mathcal{E}_C\cap\mathcal{E}_Q} \frac{h_E}{p_E}\norm{\tilde{\lambda} - \tilde{\lambda}^{p_E}}^2_{0,E}    \right].
\end{align*}
Together with 
\begin{equation*}
    \bm{E}_E(\theta^*; \tilde{u},\tilde{\lambda})
     \leq \norm{\lambda-\tilde{\lambda}}_{0,E}^2 +2 \norm{\rho}_{0,E}\norm{\gamma_C(u-\tilde{u}_{hp})}_{0,E}
 \end{equation*}
 by \eqref{lem: Equivalent Lambda}, $\L(\lambda) =  \lambda$ and Corollary~\ref{cor:localEffiE} this yields the local efficiency estimate \eqref{eq:fric:effi}.    
\end{proof}

As common for residual based a posteriori error estimators the guaranteed efficiency constant has a clear $p$-dependency. The low order term $\norm{\gamma_C(u-\tilde{u}_{hp})}_{0,E}$ in \eqref{eq:fric:effi} is undesirable but  is at least a weaker norm than $\norm{u-\tilde{u}_{hp}}_{1,\omega_Q}$. 
 We refer to Section~\ref{sec:numeric:fric} for a numerical validation of \eqref{eq:fric:effi}.
Clearly, summing over $Q\in \T_h$ the local efficiency estimate in Theorem~\ref{thm: A posteriori Friction} yields a global lower bound. {Due to not using the non-sharp trace estimate \eqref{eq: Continuity of gamma and L}, this results in a slightly better global efficiency estimate than a direct application of Theorem~\ref{effThm} would yield.}
           

    \subsection{A Mosolov Model Problem} \label{sec:Mosolov}

        The next problem we consider is a Mosolov model problem, a scalar variant of the Bingham problem. To that end let the functional $\bm{J}:H^1_0(\Omega)\to\R$ be 
         defined as
        \begin{equation}\label{eq:Mosolov:Energyfunctional}
            \bm{J}(v):= \frac12\int\limits_\Omega|\nabla v|^2\,dx+\int\limits_\Omega \rho|\nabla v|\,dx-\int\limits_\Omega fv\,dx
        \end{equation}
        with given friction coefficient $\rho\in L_+^\infty(\Omega)$ and data $f\in L^2(\Omega)$. The standard formulation of the Mosolov model problem is to find a $u\in H^1_0(\Omega)$ such that
        \begin{equation} \label{eq:Mosolov:Mini}
            \bm{J}(u)=\min\limits_{v\in H_0^1(\Omega)}\bm{J}(v).
        \end{equation}
        By \cite[Lem.~4.1]{Glowinski1984} this problem has a unique solution $u\in H^1_0(\Omega)$ and can be equivalently formulated as the following variational inequality: Find $u\in H^1_0(\Omega)$ such that
        \begin{equation}\label{eq: Bingham Variational Inequality}
            \int\limits_\Omega \nabla u \nabla(v-u)\,dx + \int\limits_\Omega \rho|\nabla v|\,dx-\int\limits_\Omega \rho|\nabla u|\,dx\geq \int\limits_\Omega f(v-u)\,dx
        \end{equation}
        for all $v\in H^1_0(\Omega)$.
        Similar to the idealized frictional problem case we may set $ V:= H^1_0(\Omega)$ and $\A$ and $\lin$ such that
        \begin{equation*}
                \langle \A v,w\rangle_V=(\nabla v,\nabla w)_{0,\Omega},\quad \langle \lin ,v\rangle_V=(f,v)_{0,\Omega}
        \end{equation*}
        for any $v,w\in H^1_0(\Omega)$. We further set $G:=\Omega$, $W:=H^1_0(\Omega)=V$ and $\gamma:=\operatorname{id}:V\to V$, which clearly fulfils \eqref{eq: Continuity of gamma and L} and that $\gamma(V)=V$ is dense in $V$. Next, we define $\L(\cdot):=\nabla(\cdot):H^1_0(\Omega)\to [L^2(\Omega)]^2$, which fulfils \eqref{eq: Continuity of gamma and L} and \eqref{eq: Lowerbound for L} by Poincar\'e Friedrich's inequality, see e.g.~\cite[Prop.~5.3.4]{Brenner2008}. This gives 
        \begin{equation*}
            j(v):=(\rho,|\nabla v|)_{0,\Omega}.
        \end{equation*}
        Note that the inequality \eqref{eq: Bingham Variational Inequality} is thus of the form \eqref{Variational Inequality}.
        To formulate the mixed version of the problem we set the operator $\bm{B}:H^1_0(\Omega)\to H^{-1}(\Omega):=(H^1_0(\Omega))^*$ such that
        \begin{equation*}
            \langle \bm{B}w, z\rangle= (\nabla w,\nabla z)_{0,\Omega}
        \end{equation*}
        for any $w,z\in H^1_0(\Omega)$, cf.~\eqref{eq: Definition of B}. Obviously, it holds that $\bm{B}=\bm{A}$ and the operator $\bm{B}$ satisfies \eqref{eq: B cont and ell}. Therefore, we have
        \begin{equation*}
            \Lambda:=\{w\in H^1_0(\Omega): (\nabla w, \nabla v)_{0,\Omega}\leq j(v)\;\text{ for all }\; v\in H^1_0(\Omega)\}.
        \end{equation*}
        This now leads to the mixed formulation of finding $u\in H^1_0(\Omega)$ and $\lambda\in \Lambda$ such that
        \begin{subequations}\label{Bingham Mixed Form}
            \begin{align}
                (\nabla u,\nabla v)_{0,\Omega} &= (f,v)_{0,\Omega}-(\nabla\lambda,\nabla v)_{0,\Omega},\\
                (\nabla (\mu&-\lambda), \nabla u)_{0,\Omega}\leq 0
            \end{align}
        \end{subequations}
        for all $v\in H^1_0(\Omega)$ and $\mu\in \Lambda$. By Theorem~\ref{Existenz Gemischte Formulierung} there exists a unique solution $(u,\lambda)\in H^1_0(\Omega)\times\Lambda$ to the mixed problem \eqref{Bingham Mixed Form} that is equivalent to the above variational inequality formulation \eqref{eq: Bingham Variational Inequality} and corresponding minimization problem \eqref{eq:Mosolov:Mini}.

        \subsubsection{A Posteriori Error Estimation} \label{sec:Mosolov:Apost}
        
            Let $\tilde{\lambda}_{hp}\in V_{hp}$ be the unique solution of
            \begin{equation} \label{eq:Mosolov:kappa_postprocessing}
                (\nabla\tilde{\lambda}_{hp},\nabla v_{hp})_{0,\Omega} = (f,v_{hp})_{0,\Omega}-(\nabla \tilde{u}_{hp},\nabla v_{hp})_{0,\Omega}
            \end{equation}
            for all $v_{hp}\in V_{hp}$. Note that compared to solving the discrete variant of \eqref{eq:Mosolov:Mini}, \eqref{eq: Bingham Variational Inequality} or \eqref{Bingham Mixed Form}, the post-processing step to acquire $\tilde{\lambda}_{hp}$ is still relatively cheap.  We once more emphasize that 
             the use of \eqref{eq:Mosolov:kappa_postprocessing}  to specify $\tilde{\lambda}$ (in the context of \eqref{eq:abstract:postprocessingLambda})
            enables the application of the classical residual error estimator for Poisson problems.  
            The auxiliary problem \eqref{Auxiliary Problem} is to find $\hat{u}\in V$ such that
            \begin{equation*}
                (\nabla \hat{u},\nabla v)_{0,\Omega} = (f,v)_{0,\Omega}-(\nabla\tilde{\lambda}_{hp},\nabla v)_{0,\Omega}
            \end{equation*}
            for all $v\in V$.
            By proceeding analogously to Section~\ref{sec:fric:apost} we obtain the total error estimator
            \begin{equation*}
                \eta^2:=\hat{\eta}^2+\bm{E}(\theta^*;\tilde{u}_{hp},\tilde{\lambda}_{hp}) 
            \end{equation*}
            where
            \begin{align*}
                \hat\eta^2&:=\sum\limits_{Q\in \T_h}\hat\eta_Q^2 \quad \text{with}\\  
                    \hat\eta_Q^2&:= \frac{h_Q^2}{p_Q^2}\norm{f^{p_Q}+\Delta (\tilde{u}_{hp}+\tilde{\lambda}_{hp})}_{0,Q}^2+ \sum\limits_{E\in \mathcal{E}_I\cap\mathcal{E}_Q}\frac{h_E}{2p_E}\norm{\llbracket \partial_{n_E}(\tilde{u}_{hp}+\tilde{\lambda}_{hp})\rrbracket}_{0,E}^2 
            \end{align*}
                and
            \begin{align*} 
                \bm{E}(\theta^*;\tilde{u}_{hp},\tilde{\lambda}_{hp}):=\norm{\theta^*-\nabla\tilde{\lambda}_{hp}}_{0,\Omega}^2+ j(\tilde{u}_{hp})-(\theta^*,\nabla \tilde{u}_{hp})_{0,\Omega}
            \end{align*}
            with
            \begin{equation*} 
                \theta^*:=\rho\frac{\nabla\tilde{\lambda}_{hp}+\frac12\nabla \tilde{u}_{hp}}{\max\{\rho,|\nabla\tilde{\lambda}_{hp}+\frac12 \nabla \tilde{u}_{hp}|\}},
            \end{equation*}
            see \eqref{eq:optimal_Prop_theta_star} and \eqref{eq:optimal_theta_star}, which lies in $\Lambda_\rho:=\{\tau\in [L^2(\Omega)]^2: |\tau|\leq \rho\}.$
As $\E$ is localizable so is $\eta^2$. Indeed, we have
\begin{align*}
\eta^2 = \sum\limits_{Q\in \T_h} \eta^2_Q \quad \text{with} \quad \eta^2_Q :=\hat{\eta}_Q^2 + \bm{E}_Q(\theta^*;\tilde{u}_{hp},\tilde{\lambda}_{hp})
\end{align*}
where
 \begin{align*} 
  \bm{E}_Q(\theta^*;\tilde{u}_{hp},\tilde{\lambda}_{hp}):=\int_Q (\theta^*-\nabla\tilde{\lambda}_{hp})^2+ \rho |\tilde{u}_{hp}| -(\theta^*)^\top \nabla \tilde{u}_{hp}\, dx.
 \end{align*}
      Furthermore, we need the data oscillation terms
            $$ \operatorname{osc}^2:=\sum\limits_{Q\in\T_h}\operatorname{osc}_Q^2 \quad \text{with} \quad  \operatorname{osc}_Q^2:= \frac{h_Q^2}{p_Q^2}\norm{f-f^{p_Q}}_{0,Q}^2. $$
            Analogously to Theorem~\ref{thm: A posteriori Friction} but exploiting that $\tilde{\lambda}_{hp} \in V_{hp}$ we obtain:
            \begin{theorem} \label{thm:Mosolov:Aposteriori}
                There exists a constant $C_{\text{rel}}>0$ independent of $h$, $p$, $\tilde{u}_{hp}$ and $\tilde{\lambda}_{hp}$ such that
                \begin{equation*}
                    \norm{u-\tilde{u}_{hp}}_{1,\Omega}^2+\norm{\lambda-\tilde{\lambda}_{hp}}_{1,\Omega}^2\leq C_{\text{rel}}\bigl[\eta^2+\operatorname{osc}^2\bigr].
                \end{equation*}
                If $\nabla \lambda\in \Lambda_\rho$,
                then for all $\varepsilon>0$ there exists a constant $C_{\text{eff}}(\varepsilon)>0$ independent of $h$, $p$, $\tilde{u}_{hp}$ and $\tilde{\lambda}_{hp}$ such that for every $Q \in \T_h$ there holds
                \begin{align*}
                    C_{\text{eff}}(\varepsilon)^{-1}\eta^2_Q&\leq p_Q^{2(1+\varepsilon)}\norm{u-\tilde{u}_{hp}}_{1,\omega_Q}^2+ p_Q^{2(1+\varepsilon)} \norm{\lambda-\tilde{\lambda}_{hp}}_{1,\omega_Q}^2+ \norm{\nabla(\lambda-\tilde{\lambda}_{hp})}_{0,Q}^2 \\
                    & \quad + \norm{\nabla(u-\tilde{u}_{hp})}_{0,Q}+ \sum_{Q \in \omega_Q} p_Q^{1+4\varepsilon}\operatorname{osc}_Q^2.
                \end{align*}
            \end{theorem}

                We note that the condition $\nabla\lambda\in\Lambda_\rho$ for the efficiency estimate  in Theorem~\ref{thm:Mosolov:Aposteriori} is a regularity assumption on $\lambda$. As the set of gradient fields $\L(\gamma(V))=\nabla H^1_0(\Omega)$ is not dense $[L^2(\Omega)]^2$, see the next section, we cannot apply Lemma~\ref{lem: Lambda weak strong equivalence} to guarantee $\nabla\lambda\in\Lambda_\rho$. However, there are examples where $\nabla\lambda\in\Lambda_\rho$ does hold, see e.g.~Section~\ref{sec:Num:Mosolov}.

 \subsubsection{An Alternative Mixed Formulation}
 The definitions of $\gamma$ and $\L$ need not be unique. In fact,  for the Mosolov problem we can interchange the role of $\gamma$ and $\L$. This leads to
$$ W:=\nabla H^1_0(\Omega) :=\left\{ w \in [L^2(\Omega)]^2: \exists v \in H^1_0(\Omega) \text{ with } \nabla v = w \right\}. $$
We note that the operator $\nabla(\cdot):H^1_0(\Omega)\to [L^2(\Omega)]^2$ is bounded from below by the Poincar\'e-Friedrich inequality and, therefore, injective and $\nabla H^1_0(\Omega)$ is closed, see~\cite[Thm.~2.5]{Abramovich2002}. Thus, $W$ is a real Hilbert space equipped with the $L^2$-norm and, consequently, $\nabla H^1_0(\Omega) \subsetneq [L^2(\Omega)]^2$ is not dense in $[L^2(\Omega)]^2$. Therefore, we set $\gamma:=\nabla(\cdot):H^1_0(\Omega)\to\nabla H^1_0(\Omega)$, which fulfils \eqref{eq: Continuity of gamma and L} and its image $\nabla H^1_0(\Omega)=W$ is dense in itself. Further, we define $\L(\cdot):=\operatorname{id}:\nabla H^1_0(\Omega)\to \nabla H^1_0(\Omega)$, which clearly satisfies \eqref{eq: Continuity of gamma and L} and \eqref{eq: Lowerbound for L}.
        In this case we have the same functional $j(\cdot)$ as before and still have the variational inequality \eqref{eq: Bingham Variational Inequality}. The operator $\B:\nabla H^1_0(\Omega)\to (\nabla H^1_0(\Omega))^*$ is simply given by the $L^2$-inner product
        \begin{equation*}
            \langle \B w,z\rangle_W = (w,z)_{0,\Omega},
        \end{equation*}
        which is both elliptic and bounded. The set $\hat\Lambda$ (taking the role of $\Lambda$) is then given by
        \begin{equation*}
           \hat \Lambda:=\{\hat{w}\in\nabla H^1_0(\Omega): (\hat{w},\nabla v)_{0,\Omega}\leq j(v)\;\text{ for all } v\in  H^1_0(\Omega)\},
        \end{equation*}
        which leads to the mixed formulation of finding $u\in H^1_0(\Omega)$ and $\hat\lambda\in \hat\Lambda$ such that
        \begin{subequations}\label{Alt Bingham Mixed Form}
            \begin{align}
                (\nabla u,\nabla v)_{0,\Omega} &= (f,v)_{0,\Omega}-(\hat\lambda,\nabla v)_{0,\Omega},\\
                (\hat \mu&-\hat\lambda, \nabla u)_{0,\Omega}\leq 0
            \end{align}
        \end{subequations}
        for all $(v,\hat\mu)\in H^1_0(\Omega) \times \hat\Lambda$. By Theorem~\ref{Existenz Gemischte Formulierung} the problem \eqref{Alt Bingham Mixed Form} is equivalent to the variational inequality \eqref{eq: Bingham Variational Inequality} and has a unique solution.

This mixed formulation is closely linked to the mixed formulation in \cite[Ch.~II]{Glowinski1984} with the only difference that there the set of admissible Lagrange multipliers is $\Lambda':=\{\hat{w}\in [L^2(\Omega)]^2: |w|\leq\rho\}$ and  not $\hat\Lambda$. This, however, implies that the Lagrange multiplier is not unique and that the inf-sup condition \eqref{lem:dense subset norm estimation} is not fulfilled , i.e.~the abstract framework of Section~\ref{sec: A Posteriori} is not applicable when $\hat{\Lambda}$ is exchanged by $\Lambda'$.

By exploiting that $\gamma:=\nabla(\cdot):H^1_0(\Omega)\to\nabla H^1_0(\Omega)$ is bijective, we easily find that \eqref{Bingham Mixed Form} and \eqref{Alt Bingham Mixed Form} are equivalent with the same primal solution $u$ and $\hat{\lambda} = \nabla \lambda$. 
Thus, applying the abstract results of Section~\ref{sec: A Posteriori} leads to an a posteriori error estimators which we could also obtain from Section~\ref{sec:Mosolov:Apost} where we consequently replace $\nabla \lambda$ with $\hat{\lambda}$.
In that case the auxiliary problem \eqref{eq:Mosolov:kappa_postprocessing} is replaced by 
                \begin{equation*}
                  \tilde{\hat\lambda} \in \nabla H^1_0(\Omega): \quad  (\tilde{\hat\lambda},\nabla v_{hp})_{0,\Omega} = (f,v_{hp})_{0,\Omega}-(\nabla \tilde{u}_{hp},\nabla v_{hp})_{0,\Omega} \quad \forall v_{hp} \in V_{hp}
                \end{equation*}
                and the functional $\E(\hat\theta^*;\tilde{u}_{hp},\tilde{\hat\lambda})$ becomes
                \begin{equation*}
                    \bm{E}(\hat\theta^*;\tilde{u}_{hp},\tilde{\hat\lambda}):=\norm{\hat\theta^*-\tilde{\hat\lambda}}_{0,\Omega}^2+ j(\tilde{u}_{hp})-(\hat\theta^*,\nabla \tilde{u}_{hp})_{0,\Omega}
                \end{equation*}
                with 
                \begin{equation*}
                    \hat\theta^*:=\rho\frac{\tilde{\hat\lambda}+\frac12\nabla \tilde{u}_{hp}}{\max\{\rho,|\tilde{\hat\lambda}+\frac12 \nabla \tilde{u}_{hp}|\}}.
                \end{equation*}
                Additionally, the local error contribution $\hat\eta_Q^2$ changes to
            \begin{equation*}
                \hat\eta_Q^2:= \frac{h_Q^2}{p_Q^2}\norm{f^{p_Q}+\Delta \tilde{u}_{hp}+\operatorname{div}\tilde{\hat\lambda}}_{0,Q}^2+ \sum\limits_{E\in \mathcal{E}_I\cap\mathcal{E}_Q}\frac{h_E}{2p_E}\norm{\llbracket \partial_{n_E}\tilde{u}_{hp}+\tilde{\hat\lambda} n_E\rrbracket}_{0,E}^2.
            \end{equation*}
            Then, we get an analogous result to Theorem~\ref{thm:Mosolov:Aposteriori} with $\norm{\lambda-\tilde{\lambda}}_{1,\Omega}^2$ replaced by $\norm{\hat{\lambda}-\tilde{\hat\lambda}}_{0,\Omega}^2$ and  $\nabla\lambda\in\Lambda_\rho$ replaced by $\hat\lambda\in\Lambda_\rho$.

 \section{Numerical Results}\label{sec: numerics}

The aim of the following numerical experiments is to validate the efficiency and reliability of the a posteriori error estimator for two different problem sets and several finite element methods going from lowest order uniform $h$-version to $hp$-adaptivity. In case of adaptive schemes we use Dörfler marking. To decide between $h$ and $p$ refinement we estimate the local Sobolev regularity based on the decay rate of the a posteriori error estimator in $p$ for the Dörfler marked elements. We refer to \cite{Banz2022} for more details on the implementation.
While the focus lies on estimating the discretization error of some Galerkin solution, we also demonstrate the flexibility and behaviour of the a posteriori error estimator by applying it to the semi-smooth Newton (SSN) iterates that converge to the sought Galerkin solution.

\subsection{An Idealized Frictional Problem} \label{sec:numeric:fric}
We reconsider the numerical example from \cite{Schroder2011a}. That is $\Omega = (-1,1)^2$, $\Gamma_D=[-1,1] \times \{1\}$, $\Gamma_C = (-1,1) \times \{-1\}$, $\Gamma_N = \partial \Omega \setminus (\Gamma_D \cup \overline{\Gamma_C})$ with the data $f \equiv -1$, $g \equiv 0$ and $\rho = 2(1-x_1)^2$. While the domain is a square and the data are polynomial functions, the solution is neither a polynomial nor smooth, see Figure~\ref{fig:fric:solution}. In fact, the solution displays six  singular points of different strength (of which only four are of numerical relevance). These are the four corner points of $\Omega$ with changes in the boundary conditions and two points on $\Gamma_C$ (the free boundary) where the solution switches between stick and slip.

\begin{figure}[ht]
  \centering 
  \begin{subfigure}[t]{0.45\textwidth}
    \centering
	\includegraphics[trim = 18mm 14mm 11mm 20mm, clip,height=35mm, keepaspectratio]{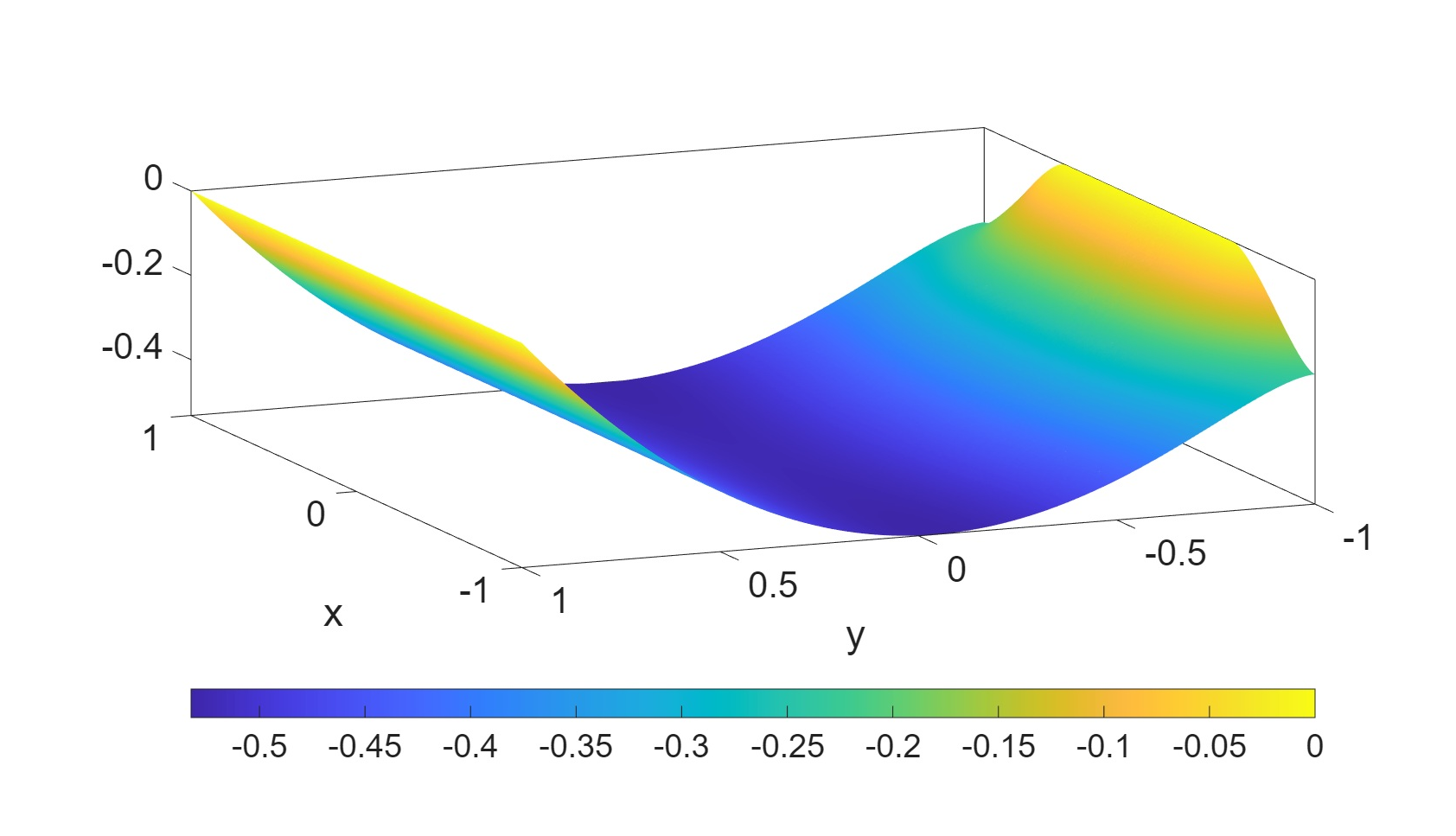}
	\caption{solution $u_{hp}$}
  \end{subfigure}
  %
  %
  \begin{subfigure}[t]{0.45\textwidth}
    \centering
	\includegraphics[trim = 20mm 9mm 20mm 9mm, clip,height=35mm, keepaspectratio]{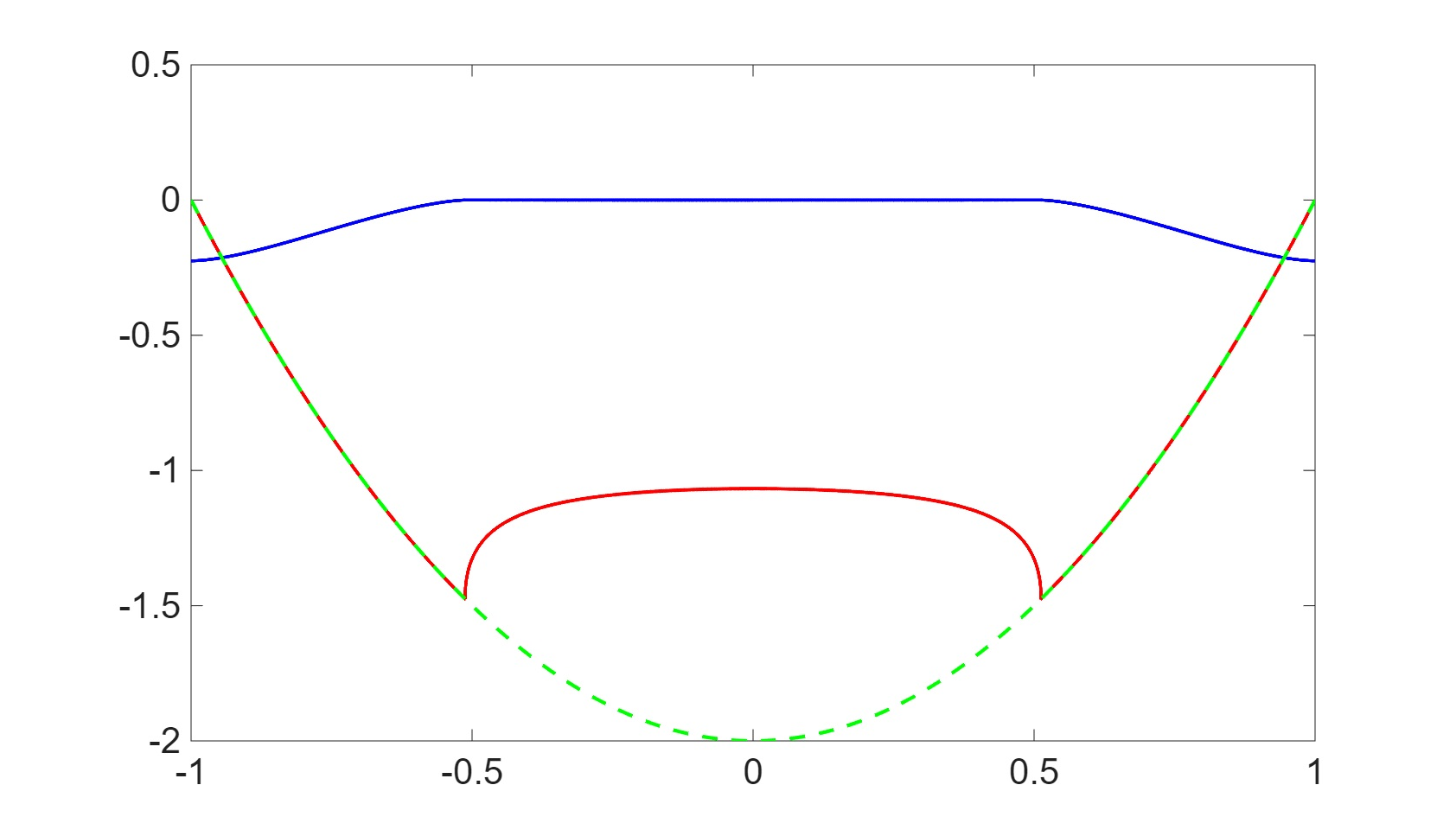}
	\caption{$u_{hp}|_{\Gamma_C}$ (blue), $\lambda_{hp}$ (red), $-\rho$ (dashed green)}
  \end{subfigure}
 \caption{Visualization of $hp$-adaptive FE solution with 24.610 + 984 degrees of freedom (idealized frictional problem).}
  \label{fig:fric:solution}
\end{figure}

We compute the discrete solution by solving the discrete mixed problem: Find $(u_{hp},\lambda_{hp}) \in V^D_{hp} \times \Lambda_{hp}$ such that
\begin{equation} \label{eq:discreteMixed}
    \aligned
      (\nabla u_{hp}, \nabla v_{hp})_{0,\Omega} &= (f,v_{hp})_{0,\Omega}+(g,\gamma(v_{hp}))_{0,\Gamma_N}-(\lambda_{hp},\gamma_C(v_{hp}))_{0,\Gamma_C},\\
      (\mu_{hp}&-\lambda_{hp},\gamma_C(u_{hp}))_{0,\Gamma_C}\leq 0
  \endaligned
\end{equation}
for all $(v_{hp},\mu_{hp}) \in V^D_{hp} \times \Lambda_{hp}$. Here, $V^D_{hp}=\operatorname{span}\{\phi_i\}_{i=1}^n$ is defined as in Section~\ref{Discretization} and spanned by Gauss-Lobatto-Lagrange basis functions $\phi_i$.  We set
$$\Lambda_{hp}:= \{w_{hp}\in \operatorname{span}\{\psi_i\}_{i=1}^m \subset L^2(\Gamma_C): (w_{hp},\gamma_C(v_{hp}))_{0,\Gamma_C}\leq j_{hp}(v_{hp}) \;\text{ for all }\; v_{hp}\in V^D_{hp}\}$$
with $$ j_{hp}(v_{hp}) := \int_{\Gamma_C} \rho(x) \sum_{i=1}^n |v_i| \phi_i(x) \, ds \quad \forall v_{hp} = \sum_{i=1}^n v_i \phi_i(x) \in V_{hp}^D.$$
The dual basis functions $\{\psi_i\}_{i=1}^m$ are biorthogonal to the $\phi_j$ with respect to $(\cdot,\cdot)_{0,\Gamma_C}$, see~e.g.~\cite{Banz2015b,Banz2015a}.  Here, $m$ is the number of primal degrees of freedom associated with $\Gamma_C$. Due to the biorthogonality of the basis functions the algebraic version of the discrete mixed problem \eqref{eq:discreteMixed} can be written as $F(\vec{u},\vec{\lambda})=0$ with a non-linear but semi-smooth function $F: \mathbb{R}^n \times \mathbb{R}^m \rightarrow \mathbb{R}^n \times \mathbb{R}^m$, see \cite{Banz2015b} for the analogous problem of Tresca friction in linear elasticity. That problem in turn is solved with a semi-smooth Newton solver that stops if $|F(\vec{u},\vec{\lambda})|^2\leq 2 \cdot 10^{-24}$, see \cite{Bammer2025Solver} for more details and a theoretical investigation of that solver in a general framework.
Following the discussion at the end of Section~\ref{sec: A Posteriori} we set $\tilde{\lambda} := \lambda_{hp}$.

As no exact solution is known we approximate the discretization error by
\begin{equation} \label{eq:fric:def_error} 
\|u-u_{hp}\|_{1,\Omega}^2 + \|\lambda - \lambda_{hp}\|_{0,\Gamma_C}^2 \approx \|u_{fine}-u_{hp}\|_{1,\Omega}^2 + \|\lambda_{fine} - \lambda_{hp}\|_{0,\Gamma_C}^2 
\end{equation}
where $(u_{fine},\lambda_{fine})$ is an overkill finite element solution of \eqref{eq:discreteMixed} which results from halving the mesh size and increasing the polynomial degree by one compared to the finest discretization of the same refinement scheme. 

\subsubsection{Numerical Results for Galerkin Solution}
Figure~\ref{fig:fric:Error} shows the reduction of the reliable a posteriori error estimator $\eta$, see Theorem~\ref{thm: A posteriori Friction}, when increasing the degrees of freedom for six different refinement strategies. These are: uniform $h$-version with $p=1,2$ and $3$; $h$-adaptive schemes with $p=2$ and $3$; as well as an $hp$-adaptive scheme.
We observe that the error estimator for the lowest order uniform $h$-version converges at optimal rate, i.e~of order $0.5$ with respect to the total degrees of freedom, i.e.~$n+m$. Increasing the polynomial degree to two or three does not increase the convergence order indicating a low Sobolev regularity of the exact solution. Switching to an $h$-adaptive scheme with a uniform polynomial degree of two or three we recover optimal order of convergence, i.e.~order $p/2$; whereas with an $hp$-adaptive scheme we even obtain exponential convergence, see~Figure~\ref{fig:fric:Error_EXPO} for an appropriate scaling of the error plot. The reason that optimal order of convergence can be achieved lies in the resolution of the four relevant point singularities using isotropic refinements, see Figure~\ref{fig:fric:meshes} for some adaptively generated meshes. This is a peculiarity of the two dimensional friction problem. For the three dimensional analogue the free boundary usually is a curve in $\Gamma_C$. Isotropic refinements  resolving that free boundary implies,  at most, an algebraic order of convergence, see also the numerical results for the Mosolov problem in Section~\ref{sec:Num:Mosolov} which also has a curved free boundary.

\begin{figure}[tb]
  \centering\hspace{-1em}
  \subfloat[Error estimator $\eta$]{ \label{fig:fric:Error}
	\begin{tikzpicture}[scale=0.82]
		\begin{loglogaxis}[
			width=0.6\textwidth,
			mark size=3pt,
			line width=0.75pt,
			xmin=1,xmax=3e6,
			ymin=1e-9,ymax=1e+1,
			legend style={at={(0,0)},anchor=south west},
            legend style={fill=none}
			]
			
			\addplot+[mark=o, color=blue] table[x index=0,y index=4] {fric_h1.txt};
 			\addplot+[mark=x, color=teal] table[x index=0,y index=4] {fric_h2.txt};
			\addplot+[mark=triangle, color=red] table[x index=0,y index=4] {fric_h3.txt};

			\addplot+[mark=square, color=orange] table[x index=0,y index=4] {fric_a2.txt};
	          \addplot+[mark=star, color=cyan, solid] table[x index=0,y index=4] {fric_a3.txt};
	        \addplot+[mark=diamond, color=purple, solid] table[x index=0,y index=4] {fric_hp.txt};
			
			\draw (5e4,3e-4) -- (5e5,3e-4);
			\draw (5e5,3e-4) -- (5e5,3e-4*0.1);
			\draw (5e4,3e-4) -- (5e5,3e-4*0.1);
			\node at (8e5,1.0e-4) {$1$};
			
			\draw (5e4,8e-7*0.031622776601684) -- (5e5,8e-7*0.031622776601684);
 			\draw (5e4,8e-7) -- (5e4,8e-7*0.031622776601684);
 			\draw (5e4,8e-7) -- (5e5,8e-7*0.031622776601684);
			\node at (2.8e4,3e-7) {$1.5$};
            
			\draw (5e4,8e-2) -- (5e5,8e-2);
			\draw (5e5,8e-2) -- (5e5,8e-2*0.316227766016838);
			\draw (5e4,8e-2) -- (5e5,8e-2*0.316227766016838);
			\node at (10e5,4.5e-2) {$0.5$};

			\legend{{$h$-unif.~$p=1$},{$h$-unif.~$p=2$},{$h$-unif.~$p=3$},{$h$-adap.~$p=2$},{$h$-adap.~$p=3$},{$hp$-adap.}}
            
		\end{loglogaxis}
	\end{tikzpicture}}    
    \hspace{0.4cm}  \subfloat[Error estimator $\eta$]{ \label{fig:fric:Error_EXPO}
	\begin{tikzpicture}[scale=0.82]
		\begin{semilogyaxis}[
		      width=0.6\textwidth,
			mark size=3pt,
			line width=0.75pt,
			xmin=1,xmax=140,
			ymin=1e-9,ymax=1e1,
			legend style={at={(1,1)},anchor=north east},
            legend style={fill=none}
			]

           	\addplot+[mark=o, color=blue] table[x index=1,y index=4] {fric_h1.txt};
 			\addplot+[mark=x, color=teal] table[x index=1,y index=4] {fric_h2.txt};
			\addplot+[mark=triangle, color=red] table[x index=1,y index=4] {fric_h3.txt};

			\addplot+[mark=square, color=orange] table[x index=1,y index=4] {fric_a2.txt};
	          \addplot+[mark=star, color=cyan, solid] table[x index=1,y index=4] {fric_a3.txt};
	        \addplot+[mark=diamond, color=purple, solid] table[x index=1,y index=4] {fric_hp.txt};

		\end{semilogyaxis}
	\end{tikzpicture}} 
	
	\caption{Error estimator vs.~total degrees of freedom $(n+m)$ (left) and vs.~$(n+m)^{1/3}$ (right) for idealized frictional problem. The legend is the same for both figures.} \label{fig:fric:error1}
\end{figure}
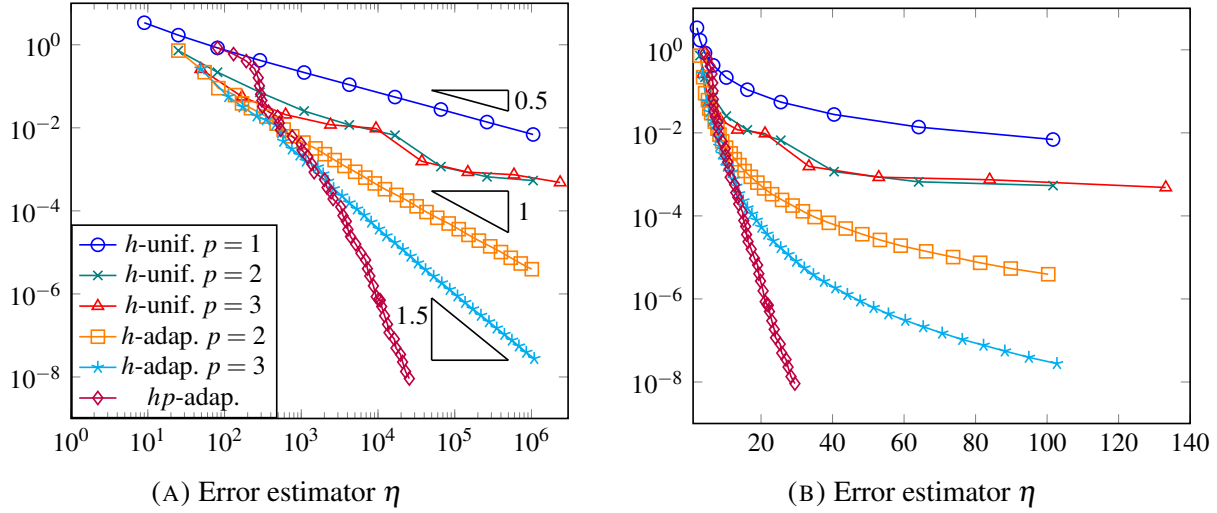

\begin{figure}
  \centering 
  \begin{subfigure}[t]{0.29\textwidth}
    \centering
   \includegraphics[trim = 80mm 18mm 70mm 12mm, clip,height=36mm, keepaspectratio]{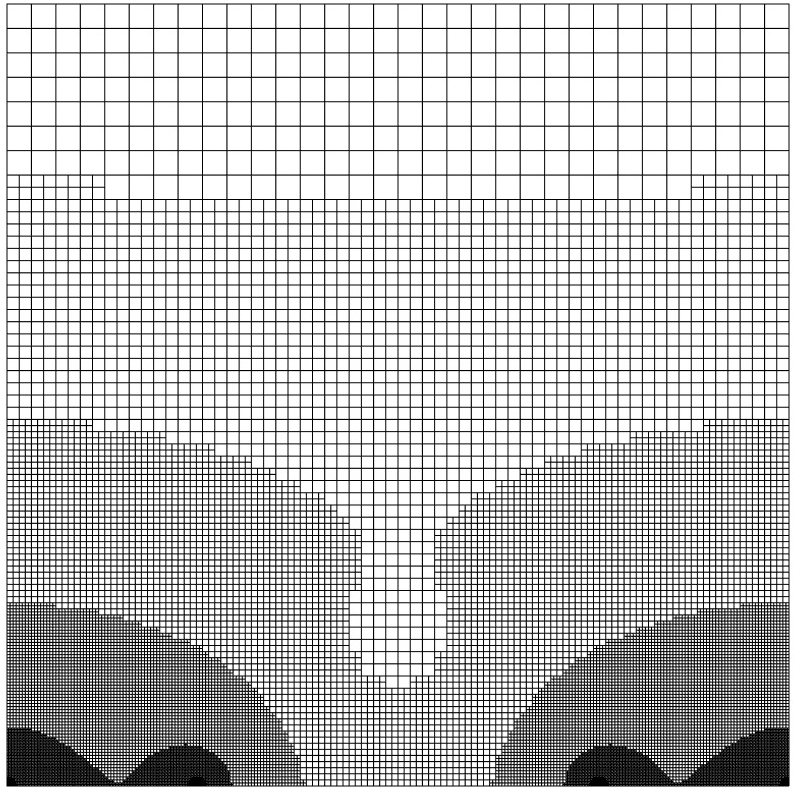}
    \caption{$h$-adaptive mesh, $p=2$ (nr.~25,$82.432+909$ DOF)}
  \end{subfigure}
  \hspace{0.2cm} 
  \begin{subfigure}[t]{0.29\textwidth}
    \centering
    \includegraphics[trim = 80mm 18mm 70mm 12mm, clip,height=36mm, keepaspectratio]{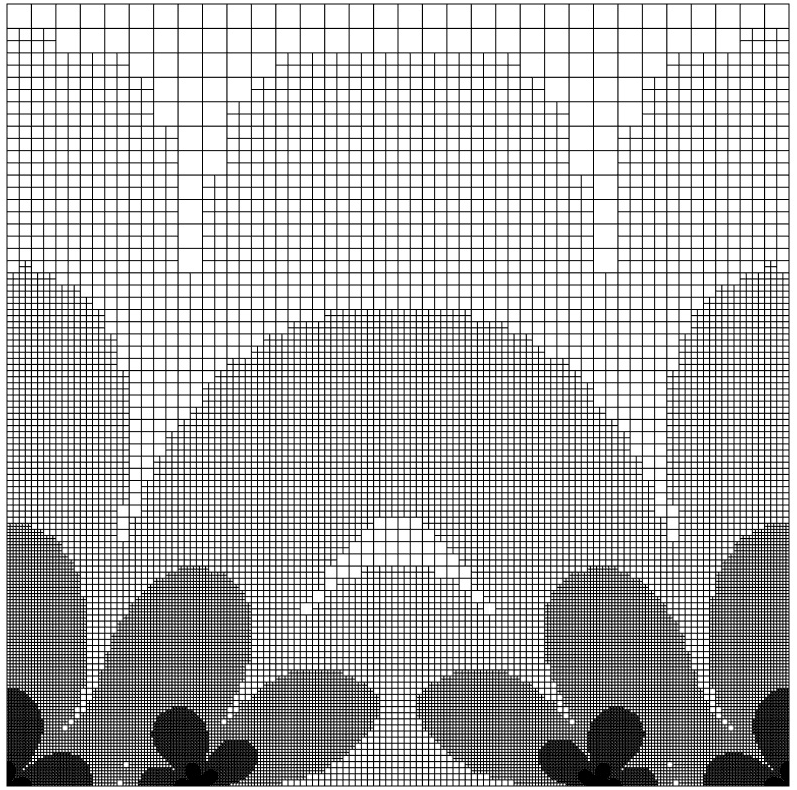}
    \caption{$h$-adaptive mesh, $p=3$ \\ (nr.~35, $277.557+1.675$ DOF)}
  \end{subfigure}
  \hspace{0.3cm} 
  \begin{subfigure}[t]{0.29\textwidth}
    \centering
    \includegraphics[trim = 80mm 14mm 43mm 10mm, clip,height=37.5mm, keepaspectratio]{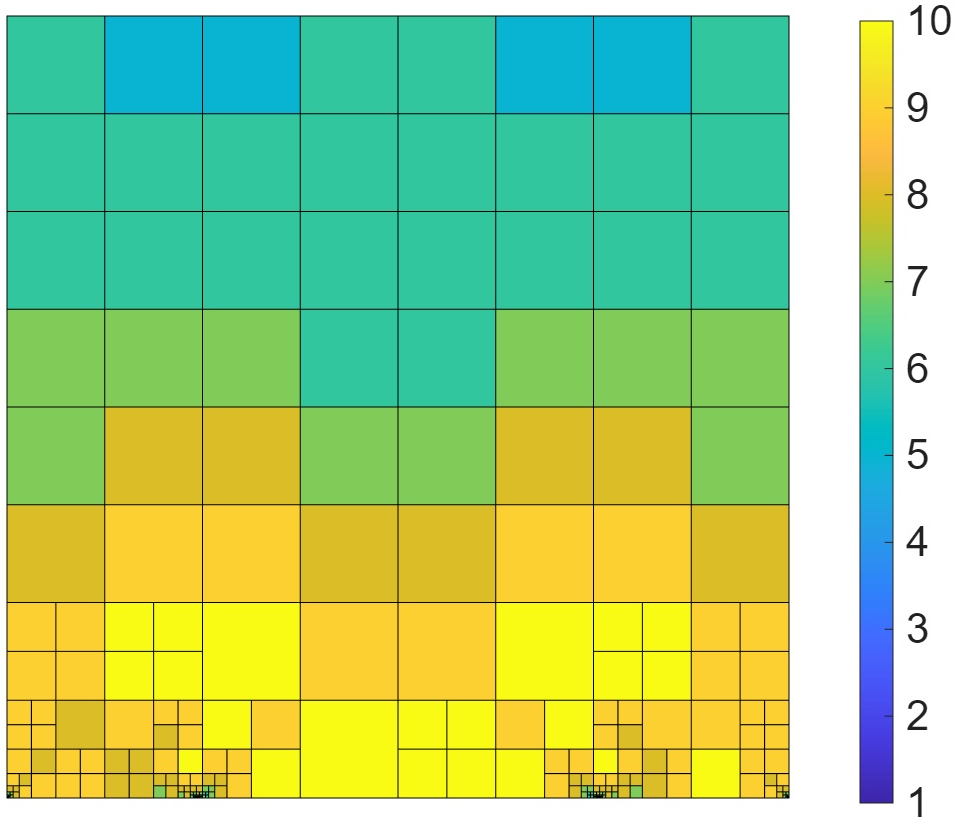}
    \caption{$hp$-adaptive mesh (nr.~47, $24.610+984$ DOF)}
  \end{subfigure}
  \hspace{0.5cm}
  \caption{\it Adaptively generated meshes for idealized frictional problem.}
  \label{fig:fric:meshes}
\end{figure}

Figure~\ref{fig:fric:effiConst} displays the efficiency index, i.e.~the error estimator $\eta$ divided by the error \eqref{eq:fric:def_error}. We observe an efficiency index ranging from 0.3 to 4, which are noticeably volatile for the adaptive schemes. Nevertheless, there is no strong upward trend indicating that the efficiency estimate of Theorem~\ref{thm: A posteriori Friction} can possibly be improved to such an extent that even the term of lower order $\|\gamma_C(u-\tilde{u}_{hp})\|_{0,E}$ can be omitted (at least for some specific Galerkin solutions).

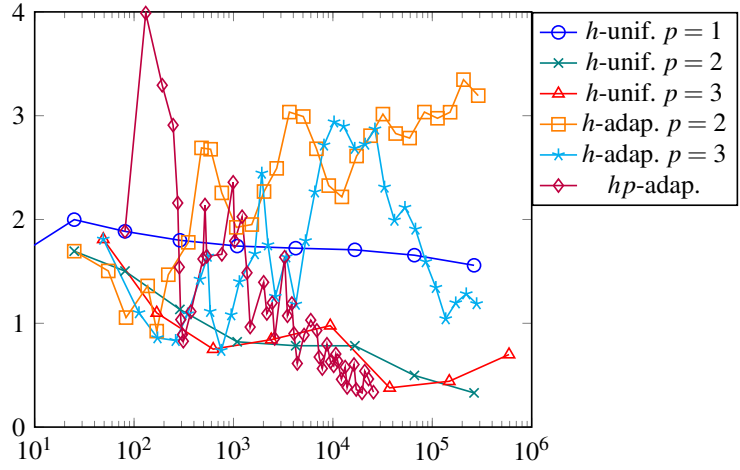
\begin{figure}[tb]
  \centering\hspace{-1em}
	\begin{tikzpicture}[scale=0.82]
		\begin{semilogxaxis}[
		      width=0.6\textwidth,
			mark size=3pt,
			line width=0.75pt,
			xmin=10,xmax=1e6,
			ymin=0,ymax=4,
			legend style={at={(1,1)},anchor=north west},
			]

			\addplot+[mark=o, color=blue] table[x index=0,y index=8] {fric_h1.txt};
 			\addplot+[mark=x, color=teal] table[x index=0,y index=8] {fric_h2.txt};
			\addplot+[mark=triangle, color=red] table[x index=0,y index=8] {fric_h3.txt};
			
			\addplot+[mark=square, color=orange] table[x index=0,y index=8] {fric_a2.txt};
	          \addplot+[mark=star, color=cyan, solid] table[x index=0,y index=8] {fric_a3.txt};
	        \addplot+[mark=diamond, color=purple, solid] table[x index=0,y index=8] {fric_hp.txt};

			\legend{{$h$-unif.~$p=1$},{$h$-unif.~$p=2$},{$h$-unif.~$p=3$},{$h$-adap.~$p=2$},{$h$-adap.~$p=3$},{$hp$-adap.}}
            
		\end{semilogxaxis}
	\end{tikzpicture}
	
	\caption{Error estimator $\eta$ divided by the error $\sqrt{\|u_{fine}-u_{hp}\|_{1,\Omega}^2 + \|\lambda_{fine}-\lambda_{hp}\|_{0,\Gamma_C}^2}$ (idealized frictional problem).} \label{fig:fric:effiConst}
\end{figure}

The behaviour of the two individual contributions of the error estimator $\eta$, namely the classical residual error estimator $\hat\eta$ and the friction specific contribution $\bm{E}^{1/2}(\theta^*; u_{hp},\lambda_{hp})$, are displayed in Figure~\ref{fig:fric:errorEstContr}.
For comparison we also plot one tenth of the primal and of the dual error, i.e.~$10^{-1}\|u_{fine}-u_{hp}\|_{1,\Omega}$ and $10^{-1}\|\lambda_{fine}-\lambda_{hp}\|_{0,\Gamma_C}$.
We find that both error estimator contributions contribute in a similar quantity. However, the friction contribution $\bm{E}^{1/2}(\theta^*; u_{hp},\lambda_{hp})$ is drawn solely from elements that have an edge on $\Gamma_C$, making that contribution vital for any marking strategy. We further gain the impression that $\bm{E}^{1/2}(\theta^*; u_{hp},\lambda_{hp})$ is more strongly linked to $\|\lambda_{fine}-\lambda_{hp}\|_{0,\Gamma_C}$, whereas $\hat\eta$ is more closely linked with
$\|u_{fine}-u_{hp}\|_{1,\Omega}$, see in particular the Figures~\ref{fig:Contribution_h2}, \ref{fig:Contribution_h3} and \ref{fig:Contribution_hp}.

\begin{figure}
  \centering 
  \begin{subfigure}[t]{0.3\textwidth}
    \centering
    \begin{tikzpicture}[scale=0.6]
		\begin{loglogaxis}[
			width=1.8\textwidth,
			mark size=2.5pt,
			line width=0.75pt,
			xmin=1,xmax=3e6,
			ymin=1e-4,ymax=1e+1,
			legend style={at={(0,0)},anchor=south west},
            legend style={fill=none}
			]
   
			\addplot+[mark=square, solid, color=blue] table[x index=0,y expr=sqrt(\thisrowno{5}^2+\thisrowno{6}^2)] {fric_h1.txt};
            \addplot+[mark=+, color=red] table[x index=0,y index=7] {fric_h1.txt};
            
            \addplot+[mark=x,solid, color=orange] table[x index=0,y expr=0.1*\thisrowno{2}] {fric_h1.txt};
            \addplot+[mark=o, solid, color=teal] table[x index=0,y expr=0.1*\thisrowno{3}] {fric_h1.txt};
      
		\end{loglogaxis}
	  \end{tikzpicture}
      \caption{$h$-uniform, $p=1$}
      \label{fig:Contribution_h1}
    \end{subfigure}  
    \hspace{0.3cm} 
  \begin{subfigure}[t]{0.3\textwidth}
    \centering
    \begin{tikzpicture}[scale=0.6]
		\begin{loglogaxis}[
			width=1.8\textwidth,
			mark size=2.5pt,
			line width=0.75pt,
			xmin=10,xmax=3e6,
			ymin=1e-5,ymax=1e+0,
			legend style={at={(0,0)},anchor=south west},
            legend style={fill=none}
			]
   
			\addplot+[mark=square, solid, color=blue] table[x index=0,y expr=sqrt(\thisrowno{5}^2+\thisrowno{6}^2)] {fric_h2.txt};
            \addplot+[mark=+, color=red] table[x index=0,y index=7] {fric_h2.txt};
            
            \addplot+[mark=x,solid, color=orange] table[x index=0,y expr=0.1*\thisrowno{2}] {fric_h2.txt};
            \addplot+[mark=o, solid, color=teal] table[x index=0,y expr=0.1*\thisrowno{3}] {fric_h2.txt};
      
		\end{loglogaxis}
	  \end{tikzpicture}
      \caption{$h$-uniform, $p=2$}
      \label{fig:Contribution_h2}
    \end{subfigure} 
     \hspace{0.3cm} 
       \begin{subfigure}[t]{0.3\textwidth}
    \centering
    \begin{tikzpicture}[scale=0.6]
		\begin{loglogaxis}[
			width=1.8\textwidth,
			mark size=2.5pt,
			line width=0.75pt,
			xmin=10,xmax=5e6,
			ymin=1e-5,ymax=1e+0,
			legend style={at={(0,0)},anchor=south west},
            legend style={fill=none}
			]
   
	       \addplot+[mark=square, solid, color=blue] table[x index=0,y expr=sqrt(\thisrowno{5}^2+\thisrowno{6}^2)] {fric_h3.txt};
            \addplot+[mark=+, color=red] table[x index=0,y index=7] {fric_h3.txt};
            
            \addplot+[mark=x,solid, color=orange] table[x index=0,y expr=0.1*\thisrowno{2}] {fric_h3.txt};
            \addplot+[mark=o, solid, color=teal] table[x index=0,y expr=0.1*\thisrowno{3}] {fric_h3.txt};
      
		\end{loglogaxis}
	  \end{tikzpicture}
      \caption{$h$-uniform, $p=3$}
      \label{fig:Contribution_h3}
    \end{subfigure}

  \begin{subfigure}[t]{0.3\textwidth}
    \centering
    \begin{tikzpicture}[scale=0.6]
		\begin{loglogaxis}[
			width=1.8\textwidth,
			mark size=2.5pt,
			line width=0.75pt,
			xmin=10,xmax=3e6,
			ymin=1e-7,ymax=1e+0,
			legend style={at={(0,0)},anchor=south west},
            legend style={fill=none}
			]
   
	       \addplot+[mark=square, solid, color=blue] table[x index=0,y expr=sqrt(\thisrowno{5}^2+\thisrowno{6}^2)] {fric_a2.txt};
            \addplot+[mark=+, color=red] table[x index=0,y index=7] {fric_a2.txt};
            
            \addplot+[mark=x,solid, color=orange] table[x index=0,y expr=0.1*\thisrowno{2}] {fric_a2.txt};
            \addplot+[mark=o, solid, color=teal] table[x index=0,y expr=0.1*\thisrowno{3}] {fric_a2.txt};
      
		\end{loglogaxis}
	  \end{tikzpicture}
      \caption{$h$-adaptive, $p=2$}
      \label{fig:Contribution_a2}
    \end{subfigure}  
    \hspace{0.3cm} 
  \begin{subfigure}[t]{0.3\textwidth}
    \centering
    \begin{tikzpicture}[scale=0.6]
		\begin{loglogaxis}[
			width=1.8\textwidth,
			mark size=2.5pt,
			line width=0.75pt,
			xmin=10,xmax=3e6,
			ymin=1e-9,ymax=1e+0,
			legend style={at={(0,0)},anchor=south west},
            legend style={fill=none}
			]
   
	        \addplot+[mark=square, solid, color=blue] table[x index=0,y expr=sqrt(\thisrowno{5}^2+\thisrowno{6}^2)] {fric_a3.txt};
            \addplot+[mark=+, color=red] table[x index=0,y index=7] {fric_a3.txt};
            
            \addplot+[mark=x,solid, color=orange] table[x index=0,y expr=0.1*\thisrowno{2}] {fric_a3.txt};
            \addplot+[mark=o, solid, color=teal] table[x index=0,y expr=0.1*\thisrowno{3}] {fric_a3.txt};
      
		\end{loglogaxis}
	  \end{tikzpicture}
      \caption{$h$-adaptive, $p=3$}
      \label{fig:Contribution_a3}
    \end{subfigure} 
     \hspace{0.3cm} 
       \begin{subfigure}[t]{0.3\textwidth}
    \centering
    \begin{tikzpicture}[scale=0.6]
		\begin{loglogaxis}[
			width=1.8\textwidth,
			mark size=2.5pt,
			line width=0.75pt,
			xmin=50,xmax=5e4,
			ymin=1e-9,ymax=1e+0,
			legend style={at={(0,0)},anchor=south west},
            legend style={fill=none}
			]
   
	        \addplot+[mark=square, solid, color=blue] table[x index=0,y expr=sqrt(\thisrowno{5}^2+\thisrowno{6}^2)] {fric_hp.txt};
            \addplot+[mark=+, color=red] table[x index=0,y index=7] {fric_hp.txt};
            
            \addplot+[mark=x,solid, color=orange] table[x index=0,y expr=0.1*\thisrowno{2}] {fric_hp.txt};
            \addplot+[mark=o, solid, color=teal] table[x index=0,y expr=0.1*\thisrowno{3}] {fric_hp.txt};

			 \legend{{residual contr.~$\hat\eta$},{friction contr.~$\bm{E}^{1/2}$},{$0.1\| u_{fine}-u_{hp}\|_{1,\Omega}$},{0.1$\| \lambda_{fine}-\lambda_{hp}\|_{0,\Gamma_C}$}}
      
		\end{loglogaxis}
	  \end{tikzpicture}
      \caption{$hp$-adaptive}
      \label{fig:Contribution_hp}
    \end{subfigure}

    \caption{Individual terms of the error estimator and of the discretization error vs.~total degrees of freedom. The legend is the same for all figures. (idealized frictional problem).}
  \label{fig:fric:errorEstContr}
    
\end{figure}
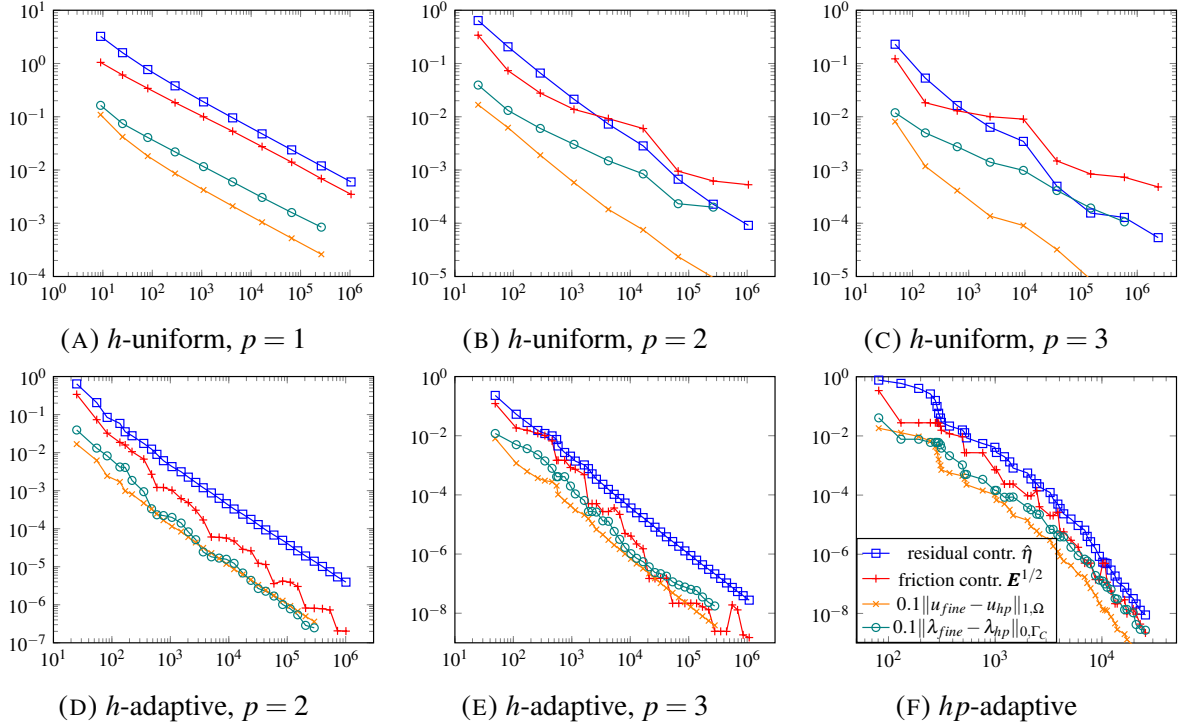

\subsubsection{Numerical Results for Semi-Smooth Newton Iterates}
To demonstrate that the a posteriori error estimator of Theorem~\ref{thm: A posteriori Friction} is applicable to a wide range of function pairs and not only to Galerkin solutions, we apply the error estimator to the SSN iterates that converge to the corresponding Galerkin solution. Let $(u_{hp}^{(i)},\lambda_{hp}^{(i)})$ be the $i$-th iterate with initial iterate $\lambda_{hp}^{(0)}=0$ and
$u_{hp}^{(0)} \in V^D_{hp}$ such that
\begin{equation*}
      (\nabla u_{hp}^{(0)}, \nabla v_{hp})_{0,\Omega} = (f,v_{hp})_{0,\Omega}+(g,\gamma(v_{hp}))_{0,\Gamma_N}
\end{equation*}
for all $v_{hp} \in V^D_{hp}$. We set $\tilde{u}_{hp}:=u_{hp}^{(i)}$ and $\tilde{\lambda} := \lambda_{hp}^{(i)}$ and note that \eqref{eq:fric:kappa_postprocessing} is fulfilled for every iterate $(u_{hp}^{(i)},\lambda_{hp}^{(i)})$ as a consequence of the initial iterate and that the first $n$ equation lines in $F(\vec{u},\vec{\lambda})=0$ are linear.

Figure~\ref{fig:fric:SSN} shows how the error estimator $\eta$ and the discretization error \eqref{eq:fric:def_error} evolve with increasing SSN iteration for different but fixed meshes. For comparison we also plot the corresponding merit values, i.e.~$2^{-1}|F(\vec{u}^{(i)},\vec{\lambda}^{(i)})|^2$, which belong to the right $y$-axes.
The error estimator curve follows nicely the error curve over all iterations but sometimes overestimates and other times underestimates the true error. The quotient of the error estimator and the true error ranges between $1/3$ and $3$, see Figure~\ref{fig:SSN_effiConst}. For these SSN iterates, the error estimator for the iterates is mainly driven by the friction contribution $\bm{E}^{1/2}(\theta^*;u^{(i)}_{hp},\lambda_{hp}^{(i)})$,  which is probably caused by the initial iterate. As that contribution is a relevant but not the largest contribution to the error estimator if applied to the Galerkin solution, see Figure~\ref{fig:fric:errorEstContr}, it explains why the efficiency index noticeably changes over the last few SSN iterations, see Figure~\ref{fig:SSN_effiConst}.
Thus, the error estimator or even the cheap to compute friction contribution alone can possibly be used for a strategy that balances out the discretization error with the solver iteration error. This is of great interest for adaptive schemes, in particular, for those coarse meshes for which the desired level of accuracy is still far from being achieved.

\begin{figure}
  \centering 
  \begin{subfigure}[t]{0.28\textwidth}
    \centering
    \begin{tikzpicture}[scale=0.55]
		\begin{semilogyaxis}[
			width=1.8\textwidth,
			mark size=2.5pt,
			line width=0.75pt,
			xmin=0,xmax=9,
			ymin=1e-3,ymax=1e+1,
            axis y line*=left,
			legend style={at={(0,0)},anchor=south west},
            legend style={fill=none}
			]
   
			\addplot+[mark=+, solid, color=blue] table[x index=0,y index=1] {SSN_h1.txt};
            \addplot+[mark=o, color=red] table[x index=0,y index=2] {SSN_h1.txt};
    		 \legend{{Error est.},{Error}}

		\end{semilogyaxis}

        \begin{semilogyaxis}[
			width=1.8\textwidth,
			mark size=2.5pt,
			line width=0.75pt,
			xmin=0,xmax=9,
			ymin=1e-30,ymax=1e+6,
            axis y line*=right,
    axis x line=none, 
			]
   
			\addplot+[mark=square, solid, color=teal] table[x index=0,y index=4] {SSN_h1.txt};

    		 \legend{{Merit value}}
      
		\end{semilogyaxis}
        
	  \end{tikzpicture}
      \caption{$h$-uniform, $p=1$ (nr.~8, $65.792+257$ DOF)}
      \label{fig:SSN_h1}
    \end{subfigure}  
    \hspace{0.3cm} 
  \begin{subfigure}[t]{0.28\textwidth}
    \centering
    \begin{tikzpicture}[scale=0.55]
		\begin{semilogyaxis}[
			width=1.8\textwidth,
			mark size=2.5pt,
			line width=0.75pt,
			xmin=0,xmax=9,
			ymin=1e-3,ymax=1e+1,
            axis y line*=left,
			legend style={at={(0,0)},anchor=south west},
            legend style={fill=none}
			]
   
			\addplot+[mark=+, solid, color=blue] table[x index=0,y index=1] {SSN_h2.txt};
            \addplot+[mark=o, color=red] table[x index=0,y index=2] {SSN_h2.txt};
    		 \legend{{Error est.},{Error}}

		\end{semilogyaxis}

        \begin{semilogyaxis}[
			width=1.8\textwidth,
			mark size=2.5pt,
			line width=0.75pt,
			xmin=0,xmax=9,
			ymin=1e-30,ymax=1e+6,
            axis y line*=right,
    axis x line=none, 
			]
   
			\addplot+[mark=square, solid, color=teal] table[x index=0,y index=4] {SSN_h2.txt};

    		 \legend{{Merit value}}
      
		\end{semilogyaxis}
	  \end{tikzpicture}
      \caption{$h$-uniform, $p=2$ (nr.~7, $65.792+257$ DOF)}
      \label{fig:SSN_h2}
    \end{subfigure} 
     \hspace{0.3cm} 
       \begin{subfigure}[t]{0.28\textwidth}
    \centering
    \begin{tikzpicture}[scale=0.55]
		\begin{semilogyaxis}[
			width=1.8\textwidth,
			mark size=2.5pt,
			line width=0.75pt,
			xmin=0,xmax=13,
			ymin=1e-5,ymax=1e+1,
            axis y line*=left,
			legend style={at={(0,0)},anchor=south west},
            legend style={fill=none}
			]
   
			\addplot+[mark=+, solid, color=blue] table[x index=0,y index=1] {SSN_a2.txt};
            \addplot+[mark=o, color=red] table[x index=0,y index=2] {SSN_a2.txt};
    		 \legend{{Error est.},{Error}}

		\end{semilogyaxis}

        \begin{semilogyaxis}[
			width=1.8\textwidth,
			mark size=2.5pt,
			line width=0.75pt,
			xmin=0,xmax=13,
			ymin=1e-30,ymax=1e+6,
            axis y line*=right,
    axis x line=none, 
			]
   
			\addplot+[mark=square, solid, color=teal] table[x index=0,y index=4] {SSN_a2.txt};

    		 \legend{{Merit value}}
      
		\end{semilogyaxis}
	  \end{tikzpicture}
      \caption{$h$-adaptive, $p=2$ (nr.~25, $82.432+909$ DOF) }
      \label{fig:SSN_a2}
    \end{subfigure}

  \begin{subfigure}[t]{0.28\textwidth}
    \centering
    \begin{tikzpicture}[scale=0.55]
		\begin{semilogyaxis}[
			width=1.8\textwidth,
			mark size=2.5pt,
			line width=0.75pt,
			xmin=0,xmax=17,
			ymin=1e-7,ymax=1e+1,
            axis y line*=left,
			legend style={at={(0,0)},anchor=south west},
            legend style={fill=none}
			]
   
			\addplot+[mark=+, solid, color=blue] table[x index=0,y index=1] {SSN_a3.txt};
            \addplot+[mark=o, color=red] table[x index=0,y index=2] {SSN_a3.txt};
    		 \legend{{Error est.},{Error}}

		\end{semilogyaxis}

        \begin{semilogyaxis}[
			width=1.8\textwidth,
			mark size=2.5pt,
			line width=0.75pt,
			xmin=0,xmax=17,
			ymin=1e-30,ymax=1e+6,
            axis y line*=right,
    axis x line=none, 
			]
   
			\addplot+[mark=square, solid, color=teal] table[x index=0,y index=4] {SSN_a3.txt};

    		 \legend{{Merit value}}
      
		\end{semilogyaxis}
	  \end{tikzpicture}
      \caption{$h$-adaptive, $p=3$ (nr.~30, $85.104+949$ DOF)}
      \label{fig:SSN_a3}
    \end{subfigure}  
    \hspace{0.3cm} 
  \begin{subfigure}[t]{0.28\textwidth}
    \centering
    \begin{tikzpicture}[scale=0.55]
		\begin{semilogyaxis}[
			width=1.8\textwidth,
			mark size=2.5pt,
			line width=0.75pt,
			xmin=0,xmax=19,
			ymin=1e-7,ymax=1e+1,
            axis y line*=left,
			legend style={at={(0,0)},anchor=south west},
            legend style={fill=none}
			]
   
			\addplot+[mark=+, solid, color=blue] table[x index=0,y index=1] {SSN_hp.txt};
            \addplot+[mark=o, color=red] table[x index=0,y index=2] {SSN_hp.txt};
    		 \legend{{Error est.},{Error}}

		\end{semilogyaxis}

        \begin{semilogyaxis}[
			width=1.8\textwidth,
			mark size=2.5pt,
			line width=0.75pt,
			xmin=0,xmax=19,
			ymin=1e-30,ymax=1e+6,
            axis y line*=right,
    axis x line=none, 
			]
   
			\addplot+[mark=square, solid, color=teal] table[x index=0,y index=4] {SSN_hp.txt};

    		 \legend{{Merit value}}
      
		\end{semilogyaxis}
	  \end{tikzpicture}
      \caption{$hp$-adaptive (nr.~40, $12.673+617$ DOF)}
      \label{fig:SSN_hp}
    \end{subfigure} 
     \hspace{0.3cm} 
       \begin{subfigure}[t]{0.28\textwidth}
    \centering
    \begin{tikzpicture}[scale=0.55]
		\begin{axis}[
			width=1.8\textwidth,
			mark size=2.5pt,
			line width=0.75pt,
			xmin=0,xmax=19,
			ymin=0,ymax=3.5,
			legend style={at={(0,1)},anchor=north west},
            legend style={fill=none}
			]
   
			\addplot+[mark=x, solid, color=blue] table[x index=0,y index=3] {SSN_h1.txt};
            \addplot+[mark=square, color=teal] table[x index=0,y index=3] {SSN_h2.txt};
            \addplot+[mark=+, color=red] table[x index=0,y index=3] {SSN_a2.txt};
            
            \addplot+[mark=o,solid, color=orange] table[x index=0,y index=3] {SSN_a3.txt};
            \addplot+[mark=diamond, solid, color=cyan] table[x index=0,y index=3] {SSN_hp.txt};


            \legend{{h1},{h2},{a2},{a3},{hp}}
      
		\end{axis}
	  \end{tikzpicture}
      \caption{Error estimator divided by error}
      \label{fig:SSN_effiConst}
    \end{subfigure}

    \caption{Error estimator and error of the SSN iterates for different but fixed meshes (idealized frictional problem).}
  \label{fig:fric:SSN}
    
\end{figure}
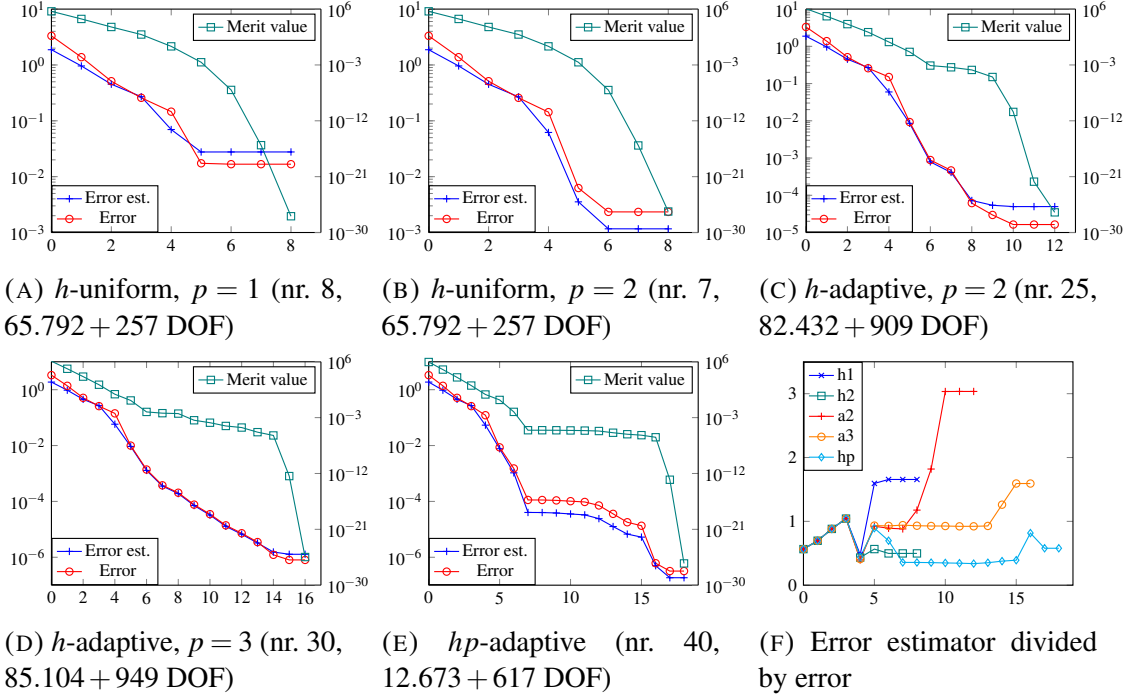


\subsection{A Mosolov Model Problem} \label{sec:Num:Mosolov}
We reconsider the numerical example from \cite{Banz2022}. That is $\Omega =B_1(0)$, $f \equiv 1$ and $\rho \equiv 0.3$ with the exact solution
\begin{align*}
 u(r,\phi)= \begin{cases}
             \frac{(1-2\rho)^2}{4},  \  & 0 \leq r \leq 2\rho\\
             \frac{(1-2\rho)^2-(r-2\rho)^2}{4},  \  & 2\rho \leq r \leq 1
            \end{cases} \quad \text{and} \quad 
 \lambda(r,\phi) = \begin{cases}
             -\frac{r^2}{4} - \rho(\rho-1),  \  & 0 \leq r \leq 2\rho\\
             -\rho(r-1),  \  & 2\rho \leq r \leq 1
            \end{cases}            
\end{align*}
given in polar coordinates, see \cite[Fig.~1]{Banz2022} for a visualization. We emphasize that $u$, $\lambda \in H^{1+3/2 - \varepsilon}(\Omega)$ for some arbitrarily small $\varepsilon>0$ with a singularity at the free boundary, which is a circle of radius $2\rho$ and centre zero. Moreover, $|\nabla \lambda| = \min\{r/2,\rho\} \leq \rho$ and, thus, the a posteriori error estimator is even locally efficient, see Theorem~\ref{thm:Mosolov:Aposteriori}.
            
The finite element solution $u_{hp}$ is computed by minimizing the non-smooth functional $\bm{J}$, see \eqref{eq:Mosolov:Energyfunctional}, over the finite element space $V_{hp}$ with a bundle Newton method that stops if the predicted reduction of the energy value is less than $10^{-20}$. The quantity $\tilde{\lambda}_{hp} \in V_{hp}$ is obtained by solving \eqref{eq:Mosolov:kappa_postprocessing} in a post-processing step.
To avoid (significant) domain approximation we use $F_Q \in [\mathbb{P}_6(\widehat{Q})]^2$  (instead of a bilinear or affine mapping) for those elements which have an edge that is a section of the boundary (i.e.~we allow elements with curved boundaries). 

Figure~\ref{fig:Mosolov:Error} shows the reduction of the total discretization error $\sqrt{\|u-u_{hp}\|_{1,\Omega}^2 + \|\lambda-\tilde{\lambda}_{hp}\|_{1,\Omega}^2}$ when increasing the degrees of freedom for seven different discretization strategies. Namely, the uniform $h$-versions with polynomial degree $p=1,2$ and $3$; the uniform $p$-version with $80$ elements; $h$-adaptive schemes with $p=2$ and $3$; as well as an $hp$-adaptive scheme. We observe optimal order of convergence for the uniform $h$-version with $p=1$ and reduced order of convergence that fits to the global Sobolev regularity of the solution pair $(u,\lambda)$ for the other three uniform schemes. The $h$-adaptive scheme with $p=2$ recovers optimal order of convergence of one with respect to the degrees of freedom, but the $h$-adaptive scheme with $p=3$ and also the $hp$-adaptive scheme only exhibit a convergence rate of $1.25$. As discussed and verified in \cite{Banz2022} this is not a failure of the a posteriori error estimator; using the exact error as error estimator yields the same behaviour.  The reason for this likely lies in the combination of an isotropic refinement strategy and the singularity behaviour of the solution at the curved free boundary.  

Here, our main focus is on the efficiency index of the error estimator, i.e.~error estimator $\eta$ divided by the total error $\sqrt{\|u-u_{hp}\|_{1,\Omega}^2 + \|\lambda-\tilde{\lambda}_{hp}\|_{1,\Omega}^2}$. The efficiency index for the seven discretization schemes under consideration seems to be constant, see Figure~\ref{fig:Mosolov:Effi}, and, thus, better than guaranteed by the theory, cf.~Theorem~\ref{thm:Mosolov:Aposteriori}. Indeed, the efficiency index tends to $8$ for the lowest order uniform $h$-version and to some value between 0.9 and 1.7 for the other higher order schemes, see the zoomed in figure within Figure~\ref{fig:Mosolov:Effi}. It is remarkable that we do not observe a $p$-dependency of the efficiency index despite using the classical residual error estimator $\hat{\eta}$ within the total error estimator $\eta$. The reason for that is, that in this example the contribution given by $\hat{\eta}$ is of much higher order (except for the lowest order uniform $h$-version where it is the dominant contribution explaining an efficiency index of 8). Thus, its usual $p$-dependency cannot become a dominant quantity. 
We refer to \cite{Banz2022} for a discussion how the individual contributions of the error estimator $\eta$ behave and for figures of the adaptively generated meshes including polynomial degree distribution.

\begin{figure}[tb]
  \centering\hspace{-1em}
  \subfloat[$\sqrt{\|u-u_{hp}\|_{1,\Omega}^2 + \|\lambda-\tilde{\lambda}_{hp}\|_{1,\Omega}^2}$]{ \label{fig:Mosolov:Error}
	\begin{tikzpicture}[scale=0.82]
		\begin{loglogaxis}[
			width=0.6\textwidth,
			mark size=3pt,
			line width=0.75pt,
			xmin=1,xmax=3e6,
			ymin=1e-6,ymax=3e-1,
			legend style={at={(0,0)},anchor=south west},
            legend style={fill=none}
			]
			
			\addplot+[mark=o, color=blue] table[x index=0,y index=1] {h1.txt};
 			\addplot+[mark=x, color=teal] table[x index=0,y index=1] {h2.txt};
			\addplot+[mark=triangle, color=red] table[x index=0,y index=1] {h3.txt};
			\addplot+[mark=+, color=darkbrown] table[x index=0,y index=1] {p.txt};

			\addplot+[mark=square, color=orange] table[x index=0,y index=1] {a2_vec.txt};
	          \addplot+[mark=star, color=cyan, solid] table[x index=0,y index=1] {a3_vec.txt};
	        \addplot+[mark=diamond, color=purple, solid] table[x index=0,y index=1] {hp_global_vec.txt};
			
			\draw (5e4,3e-4) -- (5e5,3e-4);
			\draw (5e5,3e-4) -- (5e5,3e-4*0.177827941003892);
			\draw (5e4,3e-4) -- (5e5,3e-4*0.177827941003892);
			\node at (1.1e6,1.3e-4) {$0.75$};
			
			\draw (8e3,5e-5*0.056234132519035) -- (8e4,5e-5*0.056234132519035);
 			\draw (8e3,5e-5) -- (8e3,5e-5*0.056234132519035);
 			\draw (8e3,5e-5) -- (8e4,5e-5*0.056234132519035);
			\node at (3.5e3,1e-5) {$1.25$};
            
			\draw (5e4,5e-3) -- (5e5,5e-3);
			\draw (5e5,5e-3) -- (5e5,5e-3*0.316227766016838);
			\draw (5e4,5e-3) -- (5e5,5e-3*0.316227766016838);
			\node at (9e5,2.8e-3) {$0.5$};

			\legend{{$h$-unif.~$p=1$},{$h$-unif.~$p=2$},{$h$-unif.~$p=3$},{$p$-unif.~$|\mathcal{T}_h|=80$},{$h$-adap.~$p=2$},{$h$-adap.~$p=3$},{$hp$-adap.}}
            
		\end{loglogaxis}
	\end{tikzpicture}}    
    \hspace{0.4cm}  \subfloat[Error estimator $\eta$ divided by the error $\sqrt{\|u-u_{hp}\|_{1,\Omega}^2 + \|\lambda-\tilde{\lambda}_{hp}\|_{1,\Omega}^2}$]{ \label{fig:Mosolov:Effi}
	\begin{tikzpicture}[scale=0.82]
		\begin{semilogxaxis}[
		      width=0.6\textwidth,
			mark size=3pt,
			line width=0.75pt,
			xmin=10,xmax=3e6,
			ymin=0,ymax=10,
			legend style={at={(1,1)},anchor=north east},
            legend style={fill=none}
			]

            \addplot+[mark=o, color=blue] table[x index=0,y expr=(\thisrowno{7}/\thisrowno{1})] {h1.txt};
 			\addplot+[mark=x, color=teal] table[x index=0,y expr=(\thisrowno{7}/\thisrowno{1})] {h2.txt};
			\addplot+[mark=triangle, color=red] table[x index=0,y expr=(\thisrowno{7}/\thisrowno{1})] {h3.txt};
			\addplot+[mark=+, color=darkbrown] table[x index=0,y expr=(\thisrowno{7}/\thisrowno{1})] {p.txt};

			\addplot+[mark=square, color=orange] table[x index=0,y expr=(\thisrowno{7}/\thisrowno{1})] {a2_vec.txt};
	          \addplot+[mark=star, color=cyan, solid] table[x index=0,y expr=(\thisrowno{7}/\thisrowno{1})] {a3_vec.txt};
	        \addplot+[mark=diamond, color=purple, solid] table[x index=0,y expr=(\thisrowno{7}/\thisrowno{1})] {hp_global_vec.txt};

		\end{semilogxaxis}
	\end{tikzpicture}
\put(-130,44){
    \includegraphics[trim = 25mm 130mm 110mm 105mm, clip,width=0.25\textwidth, keepaspectratio]{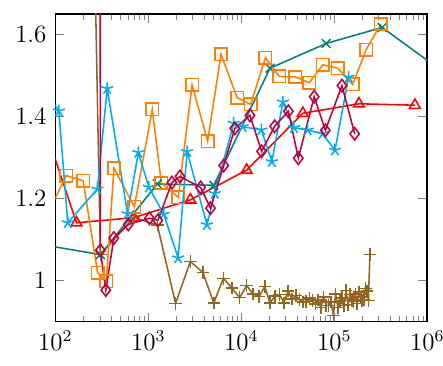}
    }
    
    } 
	
	\caption{Approximation error and efficiency index vs.~degrees of freedom (Mosolov problem). The legend is the same for both figures.} \label{fig:error1}
\end{figure}
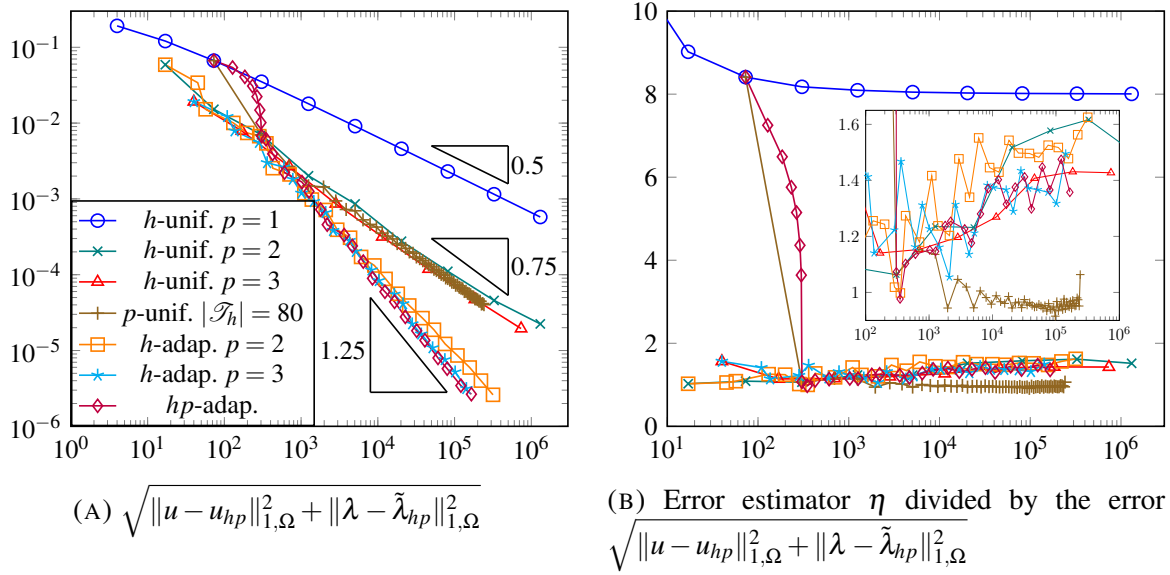

\section{Appendix} \label{app:Gemischt}

\begin{lemma}\label{lem: j properties}
    The functional $j: V \to\R$ given by
    \begin{equation*}
        j(v) := \bigl(\rho\,,|\mathfrak{L}(\gamma(v))|\bigr)_{0,G}
    \end{equation*}     
    is convex, continuous and proper.
\end{lemma}
\begin{proof}
    The linearity of $\L$ and $\gamma$, and the triangle inequality imply the convexity of $j(\cdot)$. Since $\L$, $\gamma$, $|\cdot|$ and the inner product are continuous the functional $j(\cdot)$ is also continuous. The functional $j(\cdot)$ is proper, because $j(v)>-\infty$ for all $v\in V$ and $j(0)=0<\infty$.
\end{proof}
\begin{lemma}\label{lem: Lambda properties}
    The set $\Lambda$ defined in \eqref{eq:DefLambda} is non-empty, closed and convex.
\end{lemma}
\begin{proof}
    It is clear that $0\in\Lambda$ and, thus, $\Lambda\neq\emptyset$. Furthermore, let $\zeta\in[0,1]$ and $w,z\in\Lambda$. Then it holds by \eqref{eq: Definition of B} and the linearity of $\L$ that
    \begin{equation*}
    \aligned
        \langle \B (\zeta w+(1-\zeta)z), \gamma(v)\rangle_W&=\left(\L(\zeta w+(1-\zeta)z),\L(\gamma(v))\right)_{0,G}\\
        &=\zeta\;(\L(w),\L(\gamma(v)))_{0,G}+(1-\zeta)(\L(z),\L(\gamma(v)))_{0,G}\\
        &\leq j(v),
    \endaligned
    \end{equation*}
    which means that $\Lambda$ is convex. Next, let $(w_n)_{n\geq1}\subset \Lambda$ be a sequence converging to $w\in W$. By the continuity of the $\L$ and the inner product it holds that
    \begin{equation*}
    \aligned
        \langle \B w, \gamma(v)\rangle_W &= (\lim\limits_{n\to\infty}\L(w_n),\L(\gamma(v))_{0,G}=\lim\limits_{n\to\infty}(\L(w_n),\L(\gamma(v))_{0,G}\leq j(v),
    \endaligned
    \end{equation*}
    which implies that $w\in \Lambda$. Therefore, $\Lambda$ is closed.
\end{proof}


\begin{proof}[Proof of Theorem~\ref{thm:inf-sup}]
 Let $w\in W$ be arbitrary.
    As $\gamma:V\to \gamma(V)$ is surjective and $\gamma(V)$ is a Banach space with $\norm{\cdot}_{\gamma(V)}$, it follows from \cite[Lem.~2.3]{Burg2015} that there exists a subspace $V_\gamma\subset V$ with $\gamma(V_\gamma)=\gamma(V)$ and a constant $c_\gamma>0$ such that
        \begin{equation*}
            \norm{v}_V\leq c_\gamma\,\norm{\gamma(v)}_{\gamma(V)}
        \end{equation*}
        for all $v\in V_\gamma$.
        Additionally, as $\gamma(V)$ is dense in $W$ it holds that $\norm{\B w}_{W^*}=\norm{\B w|_{\gamma(V)}}_{\gamma(V)^*}$, see e.g.~\cite[Prop.~5.4]{Stein2011}.
        Therefore, we obtain 
        \begin{equation*}
            \aligned
            \norm{\B w}_{W^*}=
            \norm{\B w\!|_{\gamma(\!V\!)}}_{\gamma\!(\!V\!)^*}&\!=\!\sup\limits_{z\in \gamma(\!V\!)\backslash\{0\}}\frac{\langle\B w,z\rangle_W}{\norm{z}_{\gamma(\!V\!)}}\\
             &\!=\! \sup\limits_{v_\gamma\in V_\gamma\backslash\{0\}}\frac{\langle\B w,\gamma(v_\gamma)\rangle_W}{\norm{\gamma(\!v_\gamma\!)}_{\gamma(V)}}\\
             &\leq c_\gamma\!\sup\limits_{v_\gamma\in V_\gamma\backslash\{0\}}\frac{\langle\B w,\gamma(\!v_\gamma\!)\rangle_W}{\norm{v_\gamma}_{V}}\\
            &\leq c_\gamma\!\sup\limits_{v\in V\backslash\{0\}}\frac{\langle\B w,\gamma(\!v\!)\rangle_W}{\norm{v}_V},
            \endaligned
        \end{equation*}
        that together with $C_b \norm{w}_W \leq  \norm{\B w}_{W^*}$, which follows from \eqref{eq: B cont and ell}, yields the assertion.
\end{proof}


\begin{proof}[Proof of Theorem \ref{Existenz Gemischte Formulierung}] For the proof we use similar arguments as in \cite[Thm.~2.2]{Burg2015}.
    Let $u\in V$ be the solution of the variational inequality \eqref{Variational Inequality}. First, we show that $\tilde{\bm{q}} \in \gamma(V)^*$ given by
        \begin{align*}
            \langle \tilde{\bm{q}}, \gamma(v)\rangle_{\gamma(V)}= \langle\bm{\ell}-\bm{A}u,v\rangle _V
        \end{align*}
    is well-defined. To that end, let $v_0,v_1\in V$ such that $\gamma(v_0)=\gamma(v_1)$.
     As $v_0-v_1+u\in V$ with $\gamma(v_0-v_1+u) = \gamma(u)$ for $u$ the solution of the variational inequality \eqref{Variational Inequality} it holds that
        \begin{align*}
            0&\leq \langle \bm{A}u-\bm{\ell}, v_0-v_1+u-u\rangle_V + j(v_0-v_1+u)-j(u)\\
            &=\langle \bm{A}u-\bm{\ell}, v_0-v_1\rangle_V + j(u) -j(u) \\
            &=\langle \bm{A}u-\bm{\ell}, v_0-v_1\rangle_V.
        \end{align*}
        Conversely, as $v_1-v_0+u\in V$ it follows in the same way that $0\leq\langle \bm{A}u-\bm{\ell}, v_1-v_0\rangle_V.$ This implies that $\langle \bm{A}u-\bm{\ell}, v_0-v_1\rangle_V =0$ and, thus,
        \begin{align*}
            \langle \tilde{\bm{q}}, \gamma(v_0)\rangle_{\gamma(V)}=\langle \bm{\ell}-\bm{A}u,v_0\rangle_V=\langle \bm{\ell}-\bm{A}u,v_1\rangle_V= \langle \tilde{\bm{q}}, \gamma(v_1)\rangle_{\gamma(V)}
        \end{align*}
        for all $v_0,v_1\in V$ with $\gamma(v_0)=\gamma(v_1)$, which gives the well-definedness of $\tilde{\bm{q}}$.\\
        Next, the density of $\gamma(V)$ in $W$ implies the existence of a unique extension $\bm{q}\in W^*$ of $\tilde{\bm{q}}$, see e.g.~\cite[Prop.~5.4]{Stein2011}.
        Consequently, it holds that
        \begin{align*}
            \langle \bm{q}, \gamma(v)\rangle_W = \langle \tilde{\bm{q}}, \gamma(v)\rangle_{\gamma(V)} = \langle\bm{\ell}-\bm{A}u,v\rangle_V
        \end{align*}
        for all $v\in V$. As the linear operator $\B$ is continuous and elliptic, there       
        there exists a unique $\lambda\in W$ such that $\B \lambda = \bm{q}$, see e.g.~\cite[Thm.~5.10]{Han2013}. Thus, it holds that
        \begin{align*}
            \langle \B \lambda, \gamma(v)\rangle_W = \langle \bm{q}, \gamma(v)\rangle_W = \langle\bm{\ell}-\bm{A}u,v\rangle_V
        \end{align*}
        for all $v\in V$, which is \eqref{Mixed Formulation Gleichung}.
        In order to show that $\lambda\in \Lambda$, let $v\in V$ be arbitrary. It holds by \eqref{Variational Inequality} and the triangle inequality that 
        \begin{align*}
            \langle \B \lambda,\gamma(v)\rangle_W &= \langle \lin-\A u,v+u-u\rangle_V
            \leq j(v+u)-j(u)
            \leq j(v)+j(u)-j(u)
            = j(v)
        \end{align*}
        and, thus, $\lambda\in\Lambda$.

        To show \eqref{eq: conditions on lambda and gamma u}, we set $v=0$ in \eqref{Variational Inequality}, which gives
        \begin{align*}
            \langle \B \lambda,\gamma(u)\rangle_W = \langle \lin -\A u, u\rangle_V \geq j(u),
        \end{align*}
        but as it also holds that $\langle \B \lambda,\gamma(u)\rangle_W\leq j(u)$ as $\lambda\in\Lambda$, we obtain
        \begin{align*}
            j(u) = \langle \B \lambda,\gamma(u)\rangle_W.
        \end{align*}
        To show \eqref{Mixed Formulation Ungleichung}, let $\mu\in\Lambda$. Then we have with \eqref{eq: conditions on lambda and gamma u} that
        \begin{align*}
            \langle \B (\mu-\lambda),\gamma(u)\rangle_W \leq j(u) - \langle \B \lambda,\gamma(u)\rangle_W = j(u)-j(u) = 0.
        \end{align*}
        Consequently, $(u,\lambda)\in V\times \Lambda$ solves the problem \eqref{Mixed Formulation}.

        For the opposite direction let us assume that $(u,{\lambda})$ is the solution of the mixed formulation \eqref{Mixed Formulation}. From $\lambda \in \Lambda$ and \eqref{Mixed Formulation Ungleichung} we obtain as above that \eqref{eq: conditions on lambda and gamma u} is true.     
        Then, as $u-v\in V$ for any $v\in V$, it holds that
        \begin{align*}
            0&= \langle \bm{A}u-\bm{\ell},v-u\rangle_V+\langle\B \lambda,\gamma(v-u)\rangle_W\\
            &= \langle\bm{A}u-\bm{\ell},v-u\rangle_V + \langle\B \lambda,\gamma(v)\rangle_W - j(u)\\
            &\leq  \langle\bm{A}u-\bm{\ell},v-u\rangle_V + j(v) - j(u).
        \end{align*}
        Therefore, $u$ is a solution of the variational inequality \eqref{Variational Inequality}.

    Due to the above proven equivalence of \eqref{Variational Inequality} and \eqref{Mixed Formulation}, the uniqueness of $u$ carries over from \eqref{Variational Inequality}. Let $(u,\lambda_1)$ and $(u,\lambda_2)$ be both solutions of \eqref{Mixed Formulation}. By subtracting the corresponding equalities in \eqref{Mixed Formulation Gleichung} we obtain
        \begin{align*}
            \langle \B(\lambda_1-\lambda_2),\gamma(v)\rangle_W = 0
        \end{align*}
        for all $v\in V$.
        The uniqueness of $\lambda$ now follows with the inf-sup condition \eqref{lem:dense subset norm estimation}.
\end{proof}

\vskip 6mm
\noindent{\bf Acknowledgments}

\noindent   
The authors gratefully acknowledge the support by the Bundesministerium f\"{u}r Frauen, Wissenschaft und Forschung (BMFWF) under the Sparkling Science project SPA 01-080 'MAJA -- Mathematische Algorithmen für Jedermann Analysiert'.

\end{document}